\documentclass[12pt,reqno]{amsart}

\usepackage{amsfonts,amsmath,latexsym,verbatim,amscd,mathrsfs,color,array}
\usepackage[colorlinks=true]{hyperref}

\usepackage{amssymb,latexsym}
\usepackage{amsmath}
\usepackage{amsthm}
\usepackage{graphicx}
\usepackage{hyperref}
\usepackage{titletoc}
\numberwithin{equation}{section}
\newtheorem{theorem}{Theorem}[section]
\newtheorem{proposition}{Proposition}[section]
\newtheorem{lemma}{Lemma}[section]

\theoremstyle{definition}

\newtheorem{remark}{Remark}

\def\XXint#1#2#3{{\setbox0=\hbox{$#1{#2#3}{\int}$}
     \vcenter{\hbox{$#2#3$}}\kern-.5\wd0}}

\def\ve{\varepsilon}

\def\R{{\mathbb R}}

\newcommand{\equ}[1]{(\ref{#1})}

\begin{document}
\title[Concentrations on segments]{Boundary concentrations on segments}

\author[W. Ao]{ Weiwei Ao}
 \address{\noindent W.Ao -Department of Mathematics, University of British
Columbia, Vancouver, B.C., Canada, V6T 1Z2  }
\email{wwao@math.ubc.ca}

\author[H. Chan]{ Hardy Chan}
 \address{\noindent H. Chan -Department of Mathematics, University of British
Columbia, Vancouver, B.C., Canada, V6T 1Z2  }
\email{hardy@math.ubc.ca}

\author[J. Wei]{Juncheng Wei}
\address{\noindent J. Wei -Department of Mathematics, University of British
Columbia, Vancouver, B.C., Canada, V6T 1Z2  }
\email{jcwei@math.ubc.ca}

\date{}\maketitle

\begin{abstract}
We consider the following singularly perturbed Neumann problem
\begin{eqnarray*}
\ve^2 \Delta u -u +u^p = 0 \ , \quad u>0 \quad {\mbox {in}} \quad \Omega, \quad
{\partial u \over \partial \nu}=0 \quad {\mbox {on}} \quad \partial \Omega,
\end{eqnarray*}
where $p>2$ and $\Omega$ is a smooth and bounded domain in $\R^2$. We construct a new class of solutions which consist of large number of spikes concentrating on a {\bf segment} of the boundary which contains a local minimum point of the mean curvature function and has the same mean curvature at the end points. We find a continuum limit of ODE systems governing the interactions of spikes and  show that the mean curvature function acts as {\em friction force}.

\end{abstract}

\setcounter{equation}{0}
\section{Introduction and statement of main results}\label{sec1}
\subsection{Introduction and Main Results}
In this paper, we establish {\bf new} concentration phenomena for the following singularly perturbed elliptic problem
\begin{equation}
\label{p}
\left\{\begin{array}{l}
-\ve^2 \Delta u +u -u^p = 0 \quad {\mbox {in}} \quad \Omega,\\
 \quad u>0 \quad {\mbox {in}} \quad \Omega,\\
{\partial u \over \partial \nu}=0 \quad {\mbox {on}} \quad \partial \Omega,
\end{array}
\right.
\end{equation}
where $\Omega$ is a smooth bounded domain in $\R^2$ with its unit outer normal $\nu$, and the exponent $p$ is greater than $2$, and $\ve>0$ is a small parameter. We prove the existence of solutions concentrating on segment of $\partial \Omega$.
\medskip

This equation is known as the stationary equation of the nonlinear Schr\"odinger equation:
\begin{equation}\label{schrodinger}
i\tilde{h}\frac{\partial \psi}{\partial t}=-\frac{\tilde{h}^2}{2m}\Delta \psi+V\psi-\tilde{\gamma}|\psi|^{p-2}\psi
\end{equation}
where $\tilde{h}$ is the Plank constant, $V$ is the potential and $\tilde{\gamma},m$ are positive constants. Then standing waves of (\ref{schrodinger}) can be found setting $\psi=e^{iEt/\tilde{h}}v(x)$ where $E$ is a constant and the real function $v$ satisfies the elliptic equation:
\begin{equation}\label{schrodinger1}
-\tilde{h}^2\Delta v+\tilde{V}v=|v|^{p-2}v
\end{equation}
for some modified potential $\tilde{V}$. If we consider $\tilde{h}\to 0$, the above equation becomes a singularly perturbed one. It can also be viewed as a stationary equation of Keller-Segel system in chemotaxis and the Gierer-Meinhardt biological pattern formation system \cite{ks}, \cite{gm}.

\medskip

Even though simple-looking, problem (\ref{p}) has a rich and interesting structure of solutions. For the last fifteen years, it has received considerable attention. In particular,  various concentration phenomena exhibited by the solutions of (\ref{p}) seem both mathematically intriguing and scientifically interesting.  We refer to three survey articles \cite{ni}, \cite{nisurvey} and \cite{weisurvey} for backgrounds and references.

\medskip

In the pioneering  papers \cite{nt1, nt2}, Ni and Takagi  proved the existence of least energy solutions to (\ref{p}), that is,  a solution $u_\epsilon$ with minimal energy. Furthermore, they showed in  \cite{nt1, nt2} that, {\it for each $\epsilon >0$ sufficiently small, $u_\epsilon$ has a spike at the most curved part of the boundary, i.e., the region where the mean curvature attains maximum value}.

\medskip

Since the publication of \cite{nt2}, problem (\ref{p}) has received  a great deal of attention and significant progress has been made. For spike solutions, solutions with multiple boundary spikes  as well as  multiple interior spikes have been established. (See \cite{amw,amw1}, \cite{bf}, \cite{bds}, \cite{dfw1}-\cite{dy}, \cite{gw1}-\cite{gpw}, \cite{li}-\cite{ln}, \cite{nw}, \cite{wboundary}, \cite{ww1}-\cite{ww2} and the references therein.) In particular, it was established in Gui and Wei \cite{gw2} that {\it for any two given integers $k\geq 0, l\geq 0$ and $ k+l >0$, problem (\ref{p}) has a solution with exactly $k$ interior spikes and $l$ boundary spikes} for every $\epsilon$ sufficiently small. Furthermore, Lin, Ni and Wei \cite{lnw} showed that there are at least $ \frac{C_N}{ (\epsilon |\log \epsilon|)^N}$ number of interior spikes.  The first author and the third author and Zeng \cite{awz} obtained the optimal bound of number of interior spikes  $\frac{C_N}{\ve^N}$ for general smooth domain in $\R^N$.

\medskip

A general principle is that for interior spike solutions, the distance function from the boundary $\partial \Omega$ plays an important role, while for the boundary spike solutions, the mean curvature function of the boundary plays an important role.

\medskip

It seems natural to ask if problem (\ref{p}) has solutions which ``concentrate'' on higher dimensional sets, e.g. curves, or surfaces.  In this regard, we  mention that it has been {\it conjectured} for a long time that {\it problem (\ref{p}) actually possesses solutions which have $m-$dimensional concentration sets for every $0 \leq m \leq N-1$}. (See e.g. \cite{ni}.)  Progress in this direction, although still limited, has also been made in \cite{amn, fm, fm1, m, mm1, mm2, mm3, mnw, wy,wy1}. For solutions concentration on interior higher dimensional sets, we mention the results in \cite{wy,wy1} where the third author and Yang constructed solutions concentrating on line segment in the interior of the domain $\Omega$. For boundary concentration solutions, we mentioned the results of Malchiodi and Montenegro \cite{mm1}-\cite{mm3} on the existence of solutions concentrating on the {\em whole boundary} or arbitrary components of $\partial \Omega$ when $\Omega\subset  \R^N$, and solutions concentrating on closed geodesics of $\partial \Omega$ when $\Omega\subset \R^3$ and also the results of Mahmoudi and Malchiodi \cite{fm} on the existence of solutions concentrating on the $k$ submanifold of $\partial \Omega\in \R^N$ provided that the sequence $\epsilon$ satisfies some gap condition. The latter condition is called {\em resonance}.

\medskip

In all the above mentioned papers, for higher dimensional boundary concentration solutions, the concentration sets are either the whole boundary or submanifold of the boundary. A natural question is:

\medskip

{\em Does problem (\ref{p}) have solutions  concentrating on a  segment of the boundary?}

\medskip

 In this paper, we give an affirmative answer to the above question. We construct solutions concentrating on an open segment $\gamma $ of the boundary $\partial \Omega\subset \R^2$, provided that  $\gamma$ satisfies the following condition:

\medskip

\noindent  $(H_1).$  Let $\gamma=\gamma([0,b])$ be a segment of the boundary $\partial \Omega$,  parametrized by arc length, and  $H(P)$  be the curvature function  of $\partial \Omega$ at $P$. Denote by
\begin{equation*}
H'(\gamma(s))= \frac{d}{ds}H(\gamma(s)), \ H''(\gamma(s))= \frac{d^2}{ds^2}H(\gamma(s)).
\end{equation*}
Assume that $H''(\gamma(s))\geq c_0>0$ for all $s\in [0,b]$, and $\int_0^b H^{'} (\gamma (s)) ds=0$.

\medskip

\begin{remark}
From assumption $(H_1)$, one can see that $\gamma$ must contain a non-degenerate local minimum point of the curvature $H$. The other condition  $\int_0^b H'(\gamma(s))ds=0$ is equivalent to $H(\gamma(0))=H(\gamma(b))$, i.e. the curvature at the two end points of $\gamma$ must be the same.
\end{remark}

Our main result in this paper can be stated  as follows:
\begin{theorem}\label{teo1}
Assume that $\gamma$ satisfies $(H_1)$, then there exists $\ve_0>0$ such that for $\ve<\ve_0$, there exists boundary spike solution to (\ref{p}) concentrating on $\gamma$.
\end{theorem}
\begin{remark}
In  \cite{gww}, Gui, Wei and Winter  proved the existence of multiple spike solutions concentrating at the local minimum point of the curvature function $H(P)$. In this paper, we have proved the existence of spike solutions concentrating on the segment which contains a local minimum of $H(P)$. Theorem {\ref{teo1}} implies that we can extend their result to a segment containing a local minimum point of $H$.
\end{remark}

\medskip

\subsection{Description of the construction}\label{sec2}
The solution we construct consists of large number ($O(\frac{1}{\ve\ln\ve})$) of spikes distributed along the segment $\gamma$ whose inter distance is sufficiently small ($O(\ve \ln\ve)$). At the first glance one may discard such kind of solutions as there seems to be no balancing force at the end points of the segment.  In the following we show that the derivative of the mean curvature function acts as {\em frictional force}. This new phenomena was first discovered in the construction of  CMC surfaces by Butscher and Mazzeo \cite{bm}. We will comment more on this in later section.

 In this subsection, we will briefly describe the solutions to be constructed later and will give the main idea in the procedure of the construction.

To be more specific, let $w$ be the unique solution of the following equation:
\begin{equation}
\left\{\begin{array}{l}
\Delta w-w+w^p=0 \mbox{ in }\R^2,\\
w>0, \ w(0)=\max_{y\in\R^2}w(y),\\
w\to 0 \mbox{ as }|y|\to \infty.
\end{array}
\right.
\end{equation}

It is well-known (see \cite{k}) that $w$ is radial, i.e., $w=w(r)$ and $w'(r)<0$ and has the following asymptotic behaviour:
\begin{equation}
w(y)=c_{N,p}|y|^{-\frac{N-1}{2}}e^{-|y|}(1+o(1))
\end{equation}
and
\begin{equation}
 w'(y)=-(1+o(1))w(y) \mbox{ as } |y|\to \infty.
\end{equation}

\medskip

For $P\in\partial \Omega$, we set
\begin{equation*}
\Omega_\ve=\{z: \ve z\in\Omega\}, \ \Omega_{\ve, P}=\{z: \ve z+P\in\Omega\},
\end{equation*}
and
\begin{equation*}
{\mathcal P} w_{P}(z)={\mathcal P}_{\Omega_{\ve,P}}w(z-\frac{P}{\ve}),  w_P(z)=w(z-\frac{P}{\ve}), \ z\in \Omega_{\ve}
\end{equation*}
where ${\mathcal P}_{\Omega_{\ve, P}}w(z-\frac{P}{\ve})$  is defined to be the unique  solution of
\begin{equation}
\Delta u-u+(w(\cdot -\frac{P}{\ve}))^p=0 \mbox{ in }\ \Omega_{\ve, P},\ \ \
\frac{\partial u}{\partial \nu}=0 \mbox{ on }\partial \Omega_{\ve, P}.
\end{equation}

\medskip

We will put large number of boundary spikes along $\gamma$. Let the location of spikes  be $(\gamma(s_1),\cdots, \gamma(s_k))$.  We define
\begin{equation*}
U=\sum_{i=1}^k {\mathcal P}_{\Omega_{\ve, \frac{\gamma(s_i)}{\ve}}}w(z-\frac{\gamma(s_i)}{\ve})
\end{equation*}
to be a approximate solution. A natural and central question is how to choose $s_i$ such that $U$ is indeed a good approximation. By formal calculation, one has the following energy expansion for the energy functional:
\begin{eqnarray*}
J(U)=\frac{k}{2}I(w)-\ve\gamma_0\sum_{i=1}^kH(\gamma(s_i))-\frac{\gamma_1}{2}w(\frac{\gamma(s_i)-\gamma(s_j)}{\ve})+o(\ve)
\end{eqnarray*}
where $\gamma_0,\gamma_1$ are positive constants.  One needs to find a critical point $(s_1,\cdots,s_k) $ of $J$ in order to get a solution of (\ref{p}), i.e. $\frac{\partial }{\partial s_i}J=0$ for all $i$. The main point in this paper is to exploit the contribution of $H'(\gamma(s))$ in $\frac{\partial J}{\partial s_i}$. The novelty of this paper is the new method of constructing balance approximate spike solutions, i.e. the configuration space $\{(s_1,\cdots,s_k)\}$, such that  $\frac{\partial J}{\partial s_i}$ is almost $0$.

\medskip

  It turns out that the number of spikes and their positions are determined by some nonlinear equations which involves the interaction of spikes and also the effect of the boundary curvature. To explain this, we need to introduce the interaction function $\Psi(s)$ which is defined for all $s\in \R$ by
\begin{equation*}
\Psi(s)=-\int_{\R^2_+}w(y-(s,0))pw^{p-1}\frac{\partial w}{\partial y_1}dy.
\end{equation*}

It turns out that $\gamma(s_i)$ are determined by the following non-linear system:
\begin{equation}\label{position}
\left\{\begin{array}{l}
\Psi(\frac{|s_2-s_1|}{\ve})+\ve^2H'(\gamma(s_1))=O(\ve^3),\\
\Psi(\frac{|s_3-s_2|}{\ve})-\Psi(\frac{|s_2-s_1|}{\ve})+\ve^2H'(\gamma(s_2))=O(\ve^3),\\
\cdots\\
\Psi(\frac{|s_k-s_{k-1}|}{\ve})-\Psi(\frac{|s_{k-1}-s_{k-2}|}{\ve})+\ve^2H'(\gamma(s_{k-1}))=O(\ve^3),\\
-\Psi(\frac{|s_k-s_{k-1}|}{\ve})+\ve^2H'(\gamma(s_k))=O(\ve^3),
\end{array}
\right.
\end{equation}
and the number of spikes depending on $\ve$ is given by $k=k_\ve=[\frac{b}{|\ve\ln\ve|}]+1$.

\medskip

In general, the above nonlinear system is difficult to solve.  Our new idea is to consider  this  non-linear system as a discretization of its continuum limiting  ODE system (as the step size tends to $0$):
\begin{equation}\label{ode1}
\left\{\begin{array}{l}
\frac{d x}{dt}=-\frac{1}{\ln\ve}\Psi^{-1}(\frac{\ve}{\ln\ve}\rho(t)),\\
\frac{d\rho}{dt}=H'(\gamma(x(t))),  \ 0<t<b_\ve,\\
\rho(0)=0, \ \rho(b_\ve)=\rho_b,\\
x'(b_\ve)=-\frac{1}{\ln\ve}\Psi^{-1}(\ve^2H'(\gamma(x(b_\ve))))
\end{array}
\right.
\end{equation}
where $\Psi^{-1} $ is the inverse function of $\Psi$, and $b_\ve=(k_\ve-1) h=b+O(h)$ and $\rho_b<0$ is a small constant depending on $\ve$. The above overdetermined ODE is solvable under the assumption of the segment $\gamma $ in $(H_1)$.

\medskip

To describe the configuration space of $\gamma(s_i)$, we solve the ODE system (\ref{ode1}) first and denote the solution as $ x(t)$. Then we define
\begin{equation}
s_i^0=x(\frac{t_i+t_{i+1}}{2}) \mbox{ for }i=1,\cdots,k-1
\end{equation}
and
\begin{equation}
s_k^0=s_{k-1}^0+\ve\Psi^{-1}(\ve^2H'(\frac{\ve}{\ln\ve}\rho_b))
\end{equation}
where
\begin{equation}
t_i=(i-1)|\ve\ln\ve|, \ i\geq 1
\end{equation}

Letting $y_i\in\R$, we define
\begin{equation}\label{points}
s_i=s_i^0+y_i, \mbox{ for } i=1,\cdots,k,
\end{equation}
and $y_i$ satisfies
\begin{equation}\label{par}
\left\{\begin{array}{l}
|y_1|\leq C|\ve\ln(-\ln\ve)|, \\
|(s_{i+1}-s_i)-(s_{i}-s_{i-1})|\leq \frac{C\ve^3}{\min \{\Psi(\frac{s_i^0-s^0_{i-1}}{\ve}),\Psi(\frac{s^0_i-s^0_{i+1}}{\ve})\}}
\end{array}
\right.
\end{equation}
for $i=2,\cdots,k-1$ for some constant $C>0$ large.

\medskip

With these notations, we can define the configuration space of $(s_1,\cdots,s_k)$ by
\begin{equation}\label{points1}
\Lambda_k=\{(s_1,\cdots,s_k)\in\R^k | s_i \mbox{ is \ defined \ by (\ref{points}) and \ satisfies \ (\ref{par})}\}
\end{equation}

\medskip

Moreover, from the analysis of the ODE (\ref{ode1}) in Section \ref{sec7},  one can get that
\begin{equation}\label{equation101}
|s_i-s_{i-1}|\geq (1+o(1))|\ve\ln\ve|, \ w(\frac{s_i-s_{i-1}}{\ve})\leq \frac{c\ve}{|\ln\ve|}
\end{equation}
for $i=2,\cdots,k$ and
\begin{equation}
|s_i-s_{i-1}|=2(1+o(1))|\ve\ln\ve||
\end{equation}
for $i=1,k$.

We will prove Theorem \ref{teo1} by showing the following result:

\medskip

\begin{theorem}
\label{teo}
 Let $\gamma$ be a segment of $\partial \Omega$ satisfying the assumption $(H_1)$. Then there exists $\ve_0$ such that for $\ve<\ve_0$, there exists positive number $k=k_{\ve,\gamma}=[\frac{b}{|\ve\ln\ve|}]$ and $k$ points $(\gamma(s_1),\cdots,\gamma(s_k))$ on $\gamma$, where $(s_1,\cdots,s_k)\in \Lambda_k$ such that there exists a solution $u_\ve$ to Problem \equ{p} and $u_\ve$ has the following form:

\begin{equation}
\label{propsol}
u_\ve (x ) =  \sum_{i=1}^k {\mathcal P}_{\Omega_{\ve, \frac{\gamma(s_i)}{\ve}} }w(\frac{x-\gamma(s_i)}{\ve})+ o(1)
\end{equation}
where $o(1) \to 0$ as $\ve \to  0$ uniformly.
\end{theorem}

\begin{remark}
The motivation of our construction comes from the study of the constant mean curvature surface. In \cite{bm}, Butscher and Mazzeo constructed CMC surface condensing to a geodesic segments by connecting large number ($O(\frac{1}{r})$) of spheres of radius $r$ distributing along the geodesic segment. In their paper, they require the symmetry condition on the geodesic segment.  In our main theorem \ref{teo1}, if we further require that $\Omega$ is symmetric, it is easy to see that $(H_1)$ can always be satisfied near the non-degenerate minimum point of the curvature $H(\gamma(s))$. Since Theorem \ref{teo1} can deal with more general segment, we believe that our idea can be used to construct CMC surface condensing to geodesic segments without the symmetry condition.  We will discuss this in a forthcoming paper. (A. Butscher announced this result in the preprint \cite{bu} but the full details have never appeared.)

\end{remark}

\subsection{Sketch of the proof of Theorem \ref{teo}}

We introduce some notations first. Since after scaling $x=\ve z$, the original problem is
\begin{equation}
\left\{\begin{array}{l}
\Delta u-u+u^p=0  \mbox{ in }\Omega_\ve\\
u>0 \mbox{ in }\Omega\\
\frac{\partial u}{\partial \nu}=0 \mbox{ on }\partial \Omega_\ve
\end{array}
\right.
\end{equation}

Fixing ${\bf s}=(s_1,\cdots,s_k)\in \Lambda_{k}$, we denote by
\begin{equation}
{\bf P}=(P_1,\cdots,P_k)=(\frac{\gamma(s_1)}{\ve},\cdots , \frac{\gamma(s_k)}{\ve})
\end{equation}
and define the sum of $k$ spikes as
\begin{equation}
U=\sum_{i=1}^k {\mathcal P}_{\Omega_{\ve,P_i}}w(z-P_i).
\end{equation}
Define the operator
\begin{equation}
S(u)=\Delta u-u+u^p.
\end{equation}
We also define the following functions as the approximate kernels
\begin{equation}
Z_i=\frac{\partial {\mathcal P}_{\Omega_{\ve,P_i}}w(z-P_i)}{\partial \tau_{P_i}} , \ i=1,\cdots,k.
\end{equation}

Using $U$ as the approximate solution, and performing the Lyapunov-Schmidt reduction, we can show that there exists $\ve_0$ such that for $
\ve<\ve_0$, we can find a $\psi$ of the following projected problem:
\begin{equation}
S(U+\psi)=\sum_{i=1}^k c_iZ_i, \ \int_{\Omega_\ve}\psi Z_i=0, \ i=1,\cdots,k,
\end{equation}
where $c_i$ are constants depending on the form of $\psi,Z_i$.

Next, we need to solve the reduced problem
\begin{equation}
c_i=0, \ i=1,\cdots,k
\end{equation}
by adjusting the points in $\Lambda_k$.

\medskip

There are two main difficulties in solving the reduced problem. First we need to control the error projection caused by $\psi$. In order to control this projection, we need to work in a weighted norm, which estimates $\psi$ locally (see Section \ref{sec5}),  and also we need a further decomposition of $\psi$ which is given in Section \ref{sec5} from where one can see why we define the configuration space of $s_i$ in (\ref{points1}). Second, we need to solve a non-linear system of the form (\ref{position}), for which we use the discretezition of the ODE equation (\ref{ode1}).

\medskip

Finally, the paper is organized as follows. Some preliminary facts and useful estimates are explained in Section \ref{sec3}. Section \ref{sec4} contains the standard Liapunov-Schmidt reduction process: we study the linearized projected problem in ref{sec4.1} first and then we solve a non-linear projected problem in \ref{sec4.2}. In Section \ref{sec5} we obtain a further asymptotic behavior of $\psi$  in $\ve$.  In Section \ref{sec6}, we derive  the reduced nonlinear system of algebraic equations for the location. Section \ref{sec7} is devoted to solve the nonlinear system.

\medskip

\noindent {\bf Acknowledgments.} The research of J. Wei is partially supported by NSERC of Canada.

\setcounter{equation}{0}
\section{Technical Analysis }\label{sec3}

In this section, we introduce a projection and derive some useful estimates.
Throughout this paper, we shall use the letter $C$ to denote a generic positive constant which may vary from term to term. By the following rescaling
\begin{equation}
x=\ve z, \ z\in \Omega_\ve:=\{\ve z\in \Omega\},
\end{equation}
and equation (\ref{p}) becomes
\begin{equation}\label{eq1}
\left\{\begin{array}{l}
\Delta u-u+u^p=0, \mbox{ in }\Omega_\ve\\
\frac{\partial u}{\partial \nu}=0 \mbox{ on }\partial \Omega_\ve.
\end{array}
\right.
\end{equation}

We denote by $\R^2_+=\{(y_1,y_2)|y_2>0\}$. Let $w$ be the unique solution of
\begin{equation}\label{w}
\left\{\begin{array}{l}
\Delta w-w+w^p=0 \mbox{ in }\R^2, \\
w>0, \ w(0)=\max_{y\in R^2}w(y), \\
w(y)\to 0 \mbox{ as }|y|\to \infty.
\end{array}
\right.
\end{equation}

Let $p\in \partial \Omega$. We can define a diffeomorphism straightening the boundary. We may assume that the inward normal to $\partial \Omega$ at $p$ is pointing in the direction of the positive $x_2$ axis. Denote $B'({R})=\{|x_1|\leq {R}\}$, and $\Omega_{1}=\Omega\cap B(p,
{R})=\{(x_1,x_2)\in B(P,{R})| x_2-p_2>\rho(x_1-p_1)\}$ where $B(p,
R)=\{x\in R^2| |x-p|<R\}$. Then since $\partial \Omega$ is smooth, we can find a constant $R$ such that $\partial \Omega$ can be represented by the graph of a smooth function $\rho_p:B'(R)\to \R$ where $\rho_p(0)=0$, and $\rho_p'(0)=0$. From now on, we omit the use of $p$ in $\rho_p$ and write $\rho$ instead if this can be done without confusion. So near $p$, $\partial \Omega$ can be represented by $(x_1, \rho(x_1))$. The curvature of $\partial \Omega$ at $p$ is
$H(p)=\rho''(0)$.  After scaling, we know that near $P=\frac{p}{\ve}$, $\partial \Omega_\ve$ can be represented by $(z_1, \ve^{-1}\rho(\ve z_1))$, where $(z_1,z_2)=\ve^{-1}(x_1,x_2)$.
By Taylor's expansion, we have the following:
\begin{equation}
\ve^{-1}\rho(\ve z_1)=\frac{1}{2}\rho''(0)\ve z_1^2+\frac{1}{6}\rho^{(3)}(0)\ve^2 z_1^3+O(\ve^3 z_1^4).
\end{equation}

Recall that for a smooth bounded domain $U$, the projection $P_U$ of $H^2(U)$ onto $\{u\in H^2(U) | \frac{\partial u}{\partial \nu}=0 \mbox{ at  } \partial U \}$ is defined as follows: For $v\in H^2(U)$, let $P_U v$ be the unique solution of the boundary value problem:
\begin{equation}
\left\{\begin{array}{l}
\Delta u-u+v^p=0, \mbox{ in }U,\\
\frac{\partial u}{\partial \nu}=0 \mbox{ on }\partial U.
\end{array}
\right.
\end{equation}

Let $h_P(z)=w(z-P)-P_{\Omega_{\ve,P}}w(z-P)$, then $h_P$ satisfies
\begin{equation}
\left\{\begin{array}{l}
\Delta h_P(z)-h_P(z)=0, \mbox{ in }\Omega_\ve,\\
\frac{\partial h_P}{\partial \nu}=\frac{\partial }{\partial \nu}w(z-P) \mbox{ on }\partial \Omega_\ve.
\end{array}
\right.
\end{equation}
For $z\in \Omega_{1,\ve}$, for $P=({\bf P}_1, {\bf P}_2)$, set now
\begin{equation}\label{y}
\left\{\begin{array}{l}
y_1=z_1-{\bf P}_1,\\
y_2=z_2-{\bf P}_2-\ve^{-1}\rho(\ve (z_1-{\bf P}_1)).
\end{array}
\right.
\end{equation}

Under this transformation, the Laplace operator and the boundary derivative operator become
\begin{eqnarray*}
&&\Delta _z=\Delta_y+\rho(\ve z_1)^2\partial_{y_2y_2}-2\rho'(\ve z_1)\partial_{y_1y_2}-\ve\rho''(\ve z_1)\partial_{y_2},\\
&&(1+\rho'(\ve z_1)^2)^{\frac{1}{2}}\frac{\partial }{\partial \nu}=\rho'(\ve z_1)\partial_{y_1}-(1+{\rho'}^2(\ve z_1))\partial_{y_2}.
\end{eqnarray*}

Let $v^{(1)}$ be the unique solution of
\begin{equation}
\left\{\begin{array}{l}
\Delta v-v=0, \mbox{ in }\R^2_+\\
\frac{\partial v}{\partial y_2}=\frac{w'}{|y|}\frac{\rho''(0)}{2}y_1^2 \mbox{ on }\partial \R^2_+,
\end{array}
\right.
\end{equation}
where $w'$ is the radial derivative of $w$, i.e. $w'=w_r(r)$, and $r=|z-P|$.

 Let $v^{(2)}$ be the unique solution of
\begin{equation*}
\left\{\begin{array}{l}
\Delta v-v-2\rho''(0)y_1\frac{\partial^2 v_1}{\partial_{y_1}\partial_{y_2}}=0  \mbox{ in }\R^2_+,\\
\frac{\partial v}{\partial y_2}=-\rho''(0)y_1\frac{\partial v_1}{\partial y_1} \mbox{ on }\partial \R^2_+.
\end{array}
\right.
\end{equation*}
Let $v^{(3)}$ be the unique solution of
\begin{equation}
\left\{\begin{array}{l}
\Delta v-v=0,\mbox{ in }\R^2_+,\\
\frac{\partial v}{\partial y_2}=\frac{w'}{|y|}\frac{1}{3}\rho^{(3)}(0)y_1^3, \mbox{ on }\partial \R^2_+.
\end{array}
\right.
\end{equation}

Note that $v^{(1)}, v^{(2)}$ are even functions in $y_1$ and $v^{(3)}$ is odd function in $y_1$. Moreover, it is easy to see that $|v_i(y)|\leq Ce^{-\mu|y|}$ for any $0<\mu<1$. Let $\chi(x)$ be a smooth cut-off function, such that $\chi(x)=1$, $x\in B(0,R_0\ve|\ln\ve|)$, and $\chi(x)=0$ for $x\in B(0,2R_0\ve|\ln\ve|)^c$ for $R_0$ large enough, and $\chi_{\ve}(z)=\chi(\ve z)$ for $z\in \Omega_\ve$. In this case, one has $w(R_0|\ln\ve|)=O(\ve^{R_0})$. Set
\begin{equation}
h_P(z)=-(\ve v_1(y)+\ve^2(v_2(y)+v_3(y)))\chi_\ve(z-P)+\ve^3\xi_P(z), \ z\in \Omega_\ve.
\end{equation}
Then we have the following estimate:
\begin{proposition}\label{pro1}
\begin{equation}
\|\xi(z)\|_{H^1(\Omega_\ve)}\leq C.
\end{equation}
\end{proposition}

Proposition \ref{pro1} was proved in \cite{ww1} by Taylor expansion and a rigorous estimate for the reminder using estimates for elliptic partial differential equations. Moreover, one can checked that $|\xi(z)|\leq Ce^{-\mu|z-P|}$ for some $0<\mu<1$.

Similarly we know from \cite{ww1} that
\begin{proposition}\label{pro2}
\begin{equation}
[\frac{\partial w}{\partial \tau_P}-\frac{\partial {\mathcal P}_{\Omega_\ve}w}{\partial \tau_P}](z-P)=\ve \eta(y)\chi_\ve(z-P)+\ve^2\eta_1(z), \ z\in \Omega_\ve,
\end{equation}
where $\eta$ is the unique solution of the following equation:
\begin{equation}
\left\{\begin{array}{l}
\Delta \eta-\eta=0 \mbox{ in }\R^2_+,\\
\frac{\partial \eta}{\partial y_2}=-\frac{1}{2}(\frac{w''}{|y|^2}-\frac{w'}{|y|^3})\rho''(0)y_1^3-\frac{w'}{|y|}\rho''(0)y_1 \mbox{ on }\partial R^2_+.
\end{array}
\right.
\end{equation}
Moreover,
\begin{equation}
\|\eta_1\|_{H^1(\Omega_\ve)}\leq C.
\end{equation}
\end{proposition}

One can observe that $\eta(y)$ is an odd function in $y_1$. It can be seen that $|\eta_i(y)|\leq Ce^{-\mu|y|}$ for some $0<\mu<1$.

Finally, let
\begin{equation}
L_0=\Delta-1+pw^{p-1}(z).
\end{equation}
We have
\begin{lemma}\label{lem1}
\begin{equation}
Ker(L_0)\cap H^2_N(\R^2_+)=span\{\frac{\partial w}{\partial y_1}\},
\end{equation}
where $H^2_N(\R^2_+)=\{u\in H^2(\R^2_+),\frac{\partial u}{\partial y_2}=0 \mbox{ on }\partial \R^2_+\}$.
\end{lemma}
\begin{proof}
See Lemma 4.2 in \cite{nt1}.
\end{proof}

Next we state a useful lemma we will always use:

\begin{lemma}\label{even}
If $|q_1-q_2|<<|q_1|$, we have the following estimate:
\begin{equation}
\int_{\R^2_+}pw(y)^{p-1}(w(y-q_1e_1)+w(y+q_2e_1))\frac{\partial w}{\partial y_1}dy=O(|q_1-q_2|w(|q_1|))
\end{equation}
as $|q_1|\to \infty$.
\end{lemma}
\begin{proof}
By the oddness of $\frac{\partial w}{\partial y_1}$ in $y_1$, one has
\begin{eqnarray*}
&&\int_{\R^2_+}pw(y)^{p-1}(w(y-q_1e_1)+w(y+q_2e_1))\frac{\partial w}{\partial y_1}dy\\
&&=\int_{\R^2_+}pw(y)^{p-1}(w(y-q_1e_1)-w(y-q_2e_1))\frac{\partial w}{\partial y_1}dy\\
&&=\int_{\R^2}pw(y)^{p-1}|\frac{\partial w}{\partial y_1}|O(w'(y-q_1e_1)|q_1-q_2|)dy\\
&&=O(|q_1-q_2|)w(|q_1|).
\end{eqnarray*}
\end{proof}

\begin{remark}
In the following sections, we will denote by $y^i=(y^i_1,y^i_2)$ to be the transformation defined by (\ref{y}) centered at the point $P_i$ and $v_i^{(j)}$ be the corresponding solutions in the expansion of $h_{P_i}$.
\end{remark}

\setcounter{equation}{0}
\section{Liapunov-Schmidt Reduction}\label{sec4}

In this section, we reduce problem (\ref{eq1}) to finite dimension by the Liapunov-Schmidt reduction method. We first introduce some notations.
Let $H^2_N(\Omega_\ve)$ be the Hilbert space defined by
\begin{equation}
H^2_N(\Omega_\ve)=\{u\in H^2(\Omega_\ve)|\frac{\partial u}{\partial \nu}=0 \mbox{ on }\partial \Omega_\ve\},
\end{equation}

Define
\begin{equation}
S(u)=\Delta u-u+u^p
\end{equation}
for $u\in H^2_N(\Omega_\ve)$. The solving equation (\ref{eq1}) is equivalent to
\begin{equation}
S(u)=0, u\in H^2_N(\Omega_\ve).
\end{equation}
To this end, we first study the linearized operator
\begin{equation*}
L_\ve(\psi):=\Delta\psi-\psi+p(\sum_{i=1}^k {\mathcal P}_{\Omega_{\ve,P_i}}w(z-P_i))^{p-1}\psi,
\end{equation*}
and define the approximate kernels to be
\begin{equation}
Z_i=\frac{\partial {\mathcal P}_{\Omega_{\ve,P_i}}w(z-P_i)}{\partial \tau_{P_i}},\end{equation}
for $i=1,\cdots, k$.

\subsection{Linear projected problem}\label{sec4.1}

We first develop a solvability theory for the linear projected problem:
\begin{equation}\label{linear}
\left\{\begin{array}{l}
L_\ve(\psi)=h+\sum_{i=1}^k c_iZ_i,\\
\int_{\Omega_\ve}\psi Z_idz=0, i=1,\cdots, k,\\
\psi\in H^2_N(\Omega_\ve)
\end{array}
\right.
\end{equation}

Given $0<\mu < 1$, consider the norms
\begin{equation}\label{star}
 \quad \| h \|_{*} =\sup_{z \in \Omega_\ve} | (\sum_{j } e^{-\mu |z-P_i| } )^{-1} h(z) |
\end{equation}
where $P_i \in \Lambda_k$ with $\Lambda_k$ defined in \equ{points1}.

\medskip

The proof of the following Proposition on linearized operator, which we postpone to the appendix, is standard.

\begin{proposition} \label{p1}
There exist positive numbers $\mu \in (0,1)$,
$\ve_0$ and $C$, such that for all $\ve \leq \ve_0$, and for any
points $P_j$, $j=1,\ldots , k$
given by \equ{points1}, there is a unique solution $(\psi , c_i  )$  to problem
\equ{linear}. Furthermore
\begin{equation}
 \|\psi \|_{*} \le C \|h\|_*  .
\label{est}\end{equation}
\end{proposition}

In the following, if $\psi$ is the unique solution given by
Proposition \ref{p1}, we set
\begin{equation}
\label{carmen1} \psi = {\mathcal A} (h).
\end{equation}
Estimate \equ{est} implies
\begin{equation}\label{carmen2}
\| {\mathcal A} (h ) \|_{*} \leq C \| h \|_{*}.
\end{equation}

\subsection{The Non-linear Projected Problem}\label{sec4.2}

We are now in the position to solve the equation:
\begin{equation}\label{nonlinear}
\left\{\begin{array}{l}
L_\ve\psi=E+N(\psi)+\sum_{i=1}^kc_iZ_i,\\
\int_{\Omega_\ve}\psi Z_i=0 \mbox{ for }i=1,\cdots, k,\\
\psi\in H^2_N(\Omega_\ve)
\end{array}
\right.
\end{equation}
where $E$ is the error of the approximate solution $U$:
\begin{eqnarray}
E&=&\Delta(\sum_{i=1}^k {\mathcal P}_{\Omega_{\ve,P_i}}w(z-P_i))-(\sum_{i=1}^k {\mathcal P}_{\Omega_{\ve,P_i}}w(z-P_i))\\
&&+(\sum_{i=1}^k {\mathcal P}_{\Omega_{\ve,P_i}}w(z-P_i))^p,\nonumber
\label{error}
\end{eqnarray}
and $N(\psi)$ is the nonlinear term:
\begin{eqnarray}
N(\psi)&=&((\sum_{i=1}^k {\mathcal P}_{\Omega_{\ve,P_i}}w(z-P_i))+\psi)^p-(\sum_{i=1}^k {\mathcal P}_{\Omega_{\ve,P_i}}w(z-P_i))^p\\
&&-p(\sum_{i=1}^k {\mathcal P}_{\Omega_{\ve,P_i}}w(z-P_i))^{p-1}\psi.\nonumber
\end{eqnarray}

We have the validity of the following result:

\begin{proposition} \label{p3}
There exist $\mu\in (0,1)$, and positive numbers $\ve_0$, $C$, such that for all $\ve \leq \ve_0$, for any
points $P_j$, $j=1,\ldots , k$
given by \equ{points1},  there is a unique solution $(\psi , c_i )$  to problem
\equ{nonlinear}. This solution depends continuously on the parameters of the construction (namely $P_j$, $j=1, \ldots , k$) and furthermore
\begin{equation}
 \|\psi \|_{*} \le C \ve.
\label{est2}\end{equation}
\end{proposition}

\medskip

\begin{proof}
The proof relies on the contraction mapping theorem in
the $\| \cdot \|_{*}$-norm above introduced. Observe that $\psi$
solves \equ{nonlinear} if and only if
\begin{equation}\label{fixed}
\psi = {\mathcal A} \left( E+N(\psi ) \right)
\end{equation}
where ${\mathcal A}$ is the operator introduced in \equ{carmen1}. In
other words, $\psi$ solves \equ{nonlinear} if and only if $\psi$ is a
fixed point for the operator
$$
T(\psi ) :={\mathcal A} \left( E+N(\psi ) \right).
$$
Given $C>0$ large enough, define
$$
{\mathcal B} = \{ \psi \in H_N^2(\Omega_\ve ) \, : \, \| \psi \|_{*} \leq
C\ve, \, \int_{\Omega_\ve} \psi Z_i
= 0 \}
$$
We will prove that $T$ is a contraction mapping from ${\mathcal B}$
in itself.

To do so, we claim that
\begin{equation}
\label{estimateE} \| E \|_{*} \leq C \ve
\end{equation}
and
\begin{equation}
\label{estimateN} \| N(\phi ) \|_{*} \leq C\| \phi \|_{*}^2
\end{equation}
for some fixed function $C$ independent of $\ve$, as $\ve \to
0$. We postpone the proof of the estimates above to the end of
the proof of this Proposition. Assuming the validity of
\equ{estimateE} and \equ{estimateN} and taking into account
\equ{carmen2}, we have for any $\psi \in {\mathcal B}$
$$
\begin{array}{ll}
\| T( \psi ) \|_{*} \leq & C\left[ \| E+N(\psi ) \|_* \right] \leq
C(\ve+\ve^2)
\\
&\leq C_1\ve
\end{array}
$$
for a proper choice of $C_1$ in the definition of ${\mathcal B}$.
Take now $\psi_1 $ and $\psi_2$ in ${\mathcal B}$. Then it is
straightforward to show that
$$
\begin{array}{ll}
\| T(\psi_1 ) - T(\psi_2 )\|_{*} &\leq C \| N(\psi_1 ) - N(\psi_2 )
\|_* \\
\\
&\leq C \left[ \| \psi_1 \|_* +  \| \psi_2
\|_* \right] \, \| \psi_1  - \psi_2  \|_* \\
\\
&\leq o(1)
\| \psi_1  - \psi_2  \|_*
\end{array}
$$
This means that $T$ is a contraction mapping from ${\mathcal B}$
into itself.

To conclude the proof of this Proposition we are left to show the
validity of \equ{estimateE} and \equ{estimateN}. We start with
\equ{estimateE}.

Fix $j \in \{ 1 , \ldots , k\}$ and consider the region $|z-P_j |\leq \frac{\min\{|P_j-P_{j-1}|,|P_j-P_{j+1}|\}}{2}$. In this region the error $E$, whose definition is in \equ{error},
can be estimated in the following way
\begin{eqnarray}\label{EE1}
|E(z) | &&\leq  C  w^{p-1} (z-P_j ) [\sum_{P_i \not= P_j}
w(z-P_i ) + \sum_i h_{P_i}(z)]
\nonumber
\\\
&&\leq  C(\ve+\ve^{\frac{p-\mu}{2}})e^{-\mu|z-P_j|}\leq C\ve e^{-\mu|z-P_j|}
\end{eqnarray}
if we choose $\mu $ small enough such that $p-\mu>2$.

\medskip

Consider now the region $|z-P_j | > \frac{\min\{|P_j-P_{j-1}|,|P_j-P_{j+1}|\}}{2}$, for
all $j $. From
the definition of $E$, we get in the region under consideration
\begin{eqnarray}
|E(z) | &\leq & C \left[ \sum_i h_{P_i}(z)+(\sum_{i=1}^kP_{\Omega_{\ve,P_i}}w(z-P_i))^p-\sum_{i}w(x-P_i)^p\right]
\nonumber \\
&\leq & C\sum_{i}e^{-\mu|z-P_i|}(\ve+\ve^{\frac{p-\mu}{2}})\nonumber\\
&\leq &C\ve \sum_{i}e^{-\mu|z-P_i|}.\label{EE2}
\end{eqnarray}
From \equ{EE1} and \equ{EE2} we get \equ{estimateE}.

\medskip

We now prove \equ{estimateN}. Let $\psi \in {\mathcal B}$. Then
\begin{eqnarray*}
\label{enne} |N(\psi )| &&\leq | ( (\sum_{i=1}^kP_{\Omega_{\ve,P_i}}w(z-P_i))+\psi )^p - (\sum_{i=1}^kP_{\Omega_\ve}w(z-P_i))^p \\
&&- p (\sum_{i=1}^kP_{\Omega_{\ve,P_i}}w(z-P_i))^{p-1} \psi |\leq C \psi^2 .
\end{eqnarray*}
Thus we have
$$
\begin{array}{ll}
|( \sum_{j} e^{-\mu |x-P_j | })^{-1} N(\psi ) | & \leq C \| \psi \|_*^2
\end{array}
$$
This gives \equ{estimateN}.

A direct consequence of the fixed point characterization of $\psi$
given above together with the fact that the error
term $E$  depends
continuously (in the
*-norm) on the parameters $ P_j$, $j=1, \ldots , k$ is that the map $(P_1,\cdots,P_k ) \to \psi $ into the space $C(\bar \Omega_\ve )$ is
continuous (in the $*$-norm).
This concludes the proof of the Proposition.

\end{proof}

\setcounter{equation}{0}
\section{Further expansion of the error} \label{sec5}

For later purpose the asymptotic behaviour of the function $\psi$ as $\ve\to 0$. This is needed to compute the neighboring interactions.

 Before we state the result, we first consider the following equation:
\begin{equation}\label{phi}
\left\{\begin{array}{l}
\Delta \phi-\phi+pw(y)^{p-1}\phi=h+d\frac{\partial w(y)}{\partial y_1}  \mbox{ in }\R^2_+,\\
\frac{\partial \phi}{\partial y_2}=0 \mbox{ on }\partial \R^2_+,\\
\int_{\R^2_+}\phi\frac{\partial w(y)}{\partial y_1}dy=0
\end{array}
\right.
\end{equation}
where $d=\frac{\int_{\R^2_+}h\frac{\partial w}{\partial y_1}}{\int_{\R^2_+}(\frac{\partial w}{\partial y_1})^2}$.
We consider the above equation in the space where $\|h\|_{**}<+\infty$, where $\|h\|_{**}=\sup_{y\in \R^2_+}|e^{\mu_1|y|}h|$ for some $0<\mu_1<1$. It is quite standard to show the solvability of the above equation and $\phi$ will satisfy the following estimate:
\begin{equation}
\|\phi\|_{**}\leq C\|h\|_{**}.
\end{equation}

Now we decompose $\psi$ as follows:
\begin{proposition}
\begin{equation}
\psi=\sum_{i={1}}^k\chi_\ve(z-P_i)\phi_i+\ve^2\psi_1,
\end{equation}
where
\begin{equation}\label{psi}
\|\psi_1\|_*\leq C,
\end{equation}
and $\phi_i=\phi_i(y^i)$ is the unique solution of
\begin{equation}
\left\{\begin{array}{l}
\Delta \phi_i-\phi_i+pw(y^i)^{p-1}\phi_i=H_i+d_i\frac{\partial w(y^i)}{\partial y_i} \mbox{ in }\R^2_+,\\
\frac{\partial \phi_i}{\partial y_2}=0 \mbox{ on }\partial \R^2_+,\\
\int_{\R^2_+}\phi_i\frac{\partial w(y^i)}{\partial y_1}dy=0
\end{array}
\right.
\end{equation}
where $d_i$ is defined such that the right hand side of the above equation is orthogonal to $\frac{\partial w(y^i)}{\partial y_1}$ in $L^2$ norm, and
\begin{equation}
H_i=-pw(y^i)^{p-1}[w(y^i-\frac{s_{i-1}-s_i}{\ve}e_1)+w(y^i-\frac{s_{i+1}-s_i}{\ve}e_1)+\ve v_i^{(1)}],
\end{equation}
for $i=2,\cdots,k-1$ and
\begin{equation}
H_1=-pw(y^1)^{p-1}[w(y^1-\frac{s_2-s_1}{\ve}e_1)+\ve v_1^{(1)}],
\end{equation}
and
\begin{equation}
H_k=-pw(y^k)^{p-1}[w(y^k-\frac{s_{k-1}-s_k}{\ve}e_1)+\ve v_k^{(1)}],
\end{equation}
and we denote
\begin{equation}
v_{i}^{(1)}=v_{P_i}^{(1)}(y^i)
\end{equation}
are the solutions obtained in Section \ref{sec2} centered at the point $P_i$.
\end{proposition}

\begin{proof}
First by the definition of $d_i$,
\begin{equation}
d_i=\int_{\R^2_+}H_i\frac{\partial w(y^i)}{\partial y_1}dy
\end{equation}
Then from Lemma \ref{even}, and the evenness of $v_i^{(1)}$ with respect to $y^i_1$,
and the definition of the configuration space (\ref{points1}), we know that  for $i=2,\cdots,k-1$
\begin{equation}
|d_i|\leq C\ve^{-1}||s_{i+1}-s_i|-|s_i-s_{i-1}|\min\{w(\frac{s_i-s_{i+1}}{\ve}),w(\frac{s_i-s_{i-1}}{\ve})\}\leq C\ve^2,
\end{equation}
and for $i=1,k$,
\begin{equation}
|d_1|=O(w(\frac{s_1-s_{2}}{\ve}))=O(\ve^2) \ and \ |d_k|=O(w(\frac{s_k-s_{k-1}}{\ve}))=O(\ve^2).
\end{equation}
Moreover, from (\ref{phi}), we have the following estimate:
\begin{equation}
\|\phi_i\|_{**}\leq C\ve \mbox{ if }p>2+\mu_1.
\end{equation}

Our strategy to estimate $\psi_1$ is to decompose $\psi_1$ into three parts and show that each of them is bounded in $\|\cdot\|_{*}$ as $\ve\to 0$. We write $\psi_1$ as
\begin{equation}
\psi_1=\psi_{11}+\psi_{12}+\psi_{13},
\end{equation}
where $\psi_{11} $ satisfies
\begin{equation}
\left\{\begin{array}{l}
\Delta \psi_{11}-\psi_{11}=0, \mbox{ in }\Omega_{\ve}\\
\frac{\partial \psi_{11}}{\partial \nu}=-\frac{1}{\ve^2}\frac{\partial \sum_{i=1}^k\chi_\ve(z-P_i)\phi_i}{\partial \nu} \mbox{ on }\partial \Omega_\ve.
\end{array}
\right.
\end{equation}
Define $\psi_{12}$ by
\begin{equation}
\psi_{12}=\frac{1}{\ve^2}\sum_{i=1}^ks_iZ_i,
\end{equation}
and $s_i$ is determined by
\begin{equation}
M(s_i)=-\int_{\Omega_\ve}(\sum_{i=1}^k\chi_\ve(z-P_i)\phi_i+\ve^2\psi_{11})Z_i.
\end{equation}
Finally define $\psi_{13}$ to be the solution of the following equation:
\begin{equation}
\left\{\begin{array}{l}
L_\ve(\psi_{13})=\frac{1}{\ve^2}L_\ve(\psi-\sum_{i=1}^k\chi_\ve(z-P_i)\phi_i-\ve^2(\psi_{11}+\psi_{12})) \mbox{ in }\Omega_\ve\\
\frac{\partial \psi_{13}}{\partial \nu}=0 \mbox{ on }\partial \Omega_\ve\\
\int_{\Omega_\ve}\psi_{13}Z_idz=0.
\end{array}
\right.
\end{equation}

Next we will estimate $\psi_{11}, \psi_{12}, \psi_{13}$ term by term.  First we estimate $g_{1\ve}=\frac{1}{\ve^2}\frac{\partial \sum_{i=1}^k\chi_\ve(z-P_i)\phi_i}{\partial \nu} $. By direct calculation
\begin{eqnarray*}
g_{1\ve}&=&\frac{1}{\ve^2}\sum_{i={1}}^k(\chi_\ve(z-P_i)\frac{\partial \phi_i}{\partial \nu}+\phi_i\frac{\partial \chi_\ve(z-P_i)}{\partial \nu})\\
&=&\frac{1}{\ve^2}\sum_{i=1}^k\ve e^{-\mu_1|y-\frac{s_i}{\ve}|}\frac{\partial \chi_\ve(z-P_i)}{\partial \nu}+O(\ve^2)\\
&=&O(\ve^{-2}e^{-(\mu_1-\mu)R_0|\ln\ve|})\sum_{i=1}^ke^{-\mu|z-P_i|}\\
&\leq& C\sum_{i=1}^ke^{-\mu|z-P_i|}
\end{eqnarray*}
if we choose $\mu_1>\mu$ and the cutoff function in such a way that $(\mu_1-\mu)R_0\geq 1$. In the above estimate, we use the definition of $\phi_i$ and the Neumann boundary satisfied by it and the definition of the cut-off function $\chi$. Thus we have that $\|g_{1\ve}\|_*\leq C$, therefore, there exists constant $C>0$, such that
\begin{equation}
\|\psi_{11}\|_*\leq C.
\end{equation}

By the definition of $\psi_{12}$, $\phi_i$ and the estimate on $\psi_{11}$, one can obtain that
\begin{eqnarray*}
&&\int_{\Omega_\ve}(\sum_{j=1}^k\chi_\ve(z-P_j)\phi_j+\ve^2\psi_{11})Z_idz\\
&&=\int_{\Omega_\ve}\chi_\ve(z-P_i)\phi_iZ_i dz+\sum_{j=i-1,i+1}\chi_\ve(z-p_j)\phi_jZ_idz\\
&&+O(\ve^{1+(1+\mu)(1+o(1))})+O(\ve^2)
\end{eqnarray*}

In order to estimate the above term, we first consider a general function which is the solution of the following equation:
\begin{equation}
\left\{\begin{array}{l}
\Delta \phi-\phi+pw(y)^{p-1}\phi\\
=-pw(y)^{p-1}(w(y-q_1e_1)+w(y+q_2e_1)+\ve v^{(1)})+d\frac{\partial w(y)}{\partial y_1}  \mbox{ in }\R^2_+,\\
\frac{\partial \phi}{\partial y_2}=0 \mbox{ on }\partial \R^2_+,\\
\int_{\R^2_+}\phi\frac{\partial w(y)}{\partial y_1}dy=0
\end{array}
\right.
\end{equation}

We can decompose it as
\begin{equation}
\phi=\phi^1+\phi^2,
\end{equation}
where
\begin{equation}
\left\{\begin{array}{l}
\Delta \phi^1-\phi^1+pw(y)^{p-1}\phi^1\\
=-pw(y)^{p-1}(w(y-q_1e_1)+w(y+q_1e_1)+\ve v^{(1)})+d_1\frac{\partial w(y)}{\partial y_1}  \mbox{ in }\R^2_+,\\
\frac{\partial \phi^1}{\partial y_2}=0 \mbox{ on }\partial \R^2_+,\\
\int_{\R^2_+}\phi^1\frac{\partial w(y)}{\partial y_1}dy=0
\end{array}
\right.
\end{equation}
and
\begin{equation}
\left\{\begin{array}{l}
\Delta \phi^2-\phi^2+pw(y)^{p-1}\phi^2\\
=-pw(y)^{p-1}(w(y+q_2e_1)-w(y+q_1e_1))+d_2\frac{\partial w(y)}{\partial y_1}  \mbox{ in }\R^2_+,\\
\frac{\partial \phi^2}{\partial y_2}=0 \mbox{ on }\partial \R^2_+,\\
\int_{\R^2_+}\phi^2\frac{\partial w(y)}{\partial y_1}dy=0
\end{array}
\right.
\end{equation}
where $d_i$ are defined such that the right hand sides of the above equations are orthogonal to $\frac{\partial w}{\partial y_1}$ in $L^2$ norm
It is easy to see that $\phi^1$ is even in $y_1$ and by Lemma \ref{even},  $\phi^2$ satisfies
\begin{equation*}
\|\phi^2\|_{**}\leq Cw(q_1)|q_1-q_2|,
\end{equation*}
if $|q_1-q_2|<<|q_1|$ and $|q_1|\to \infty$.

Using the above estimates, we can decompose $\phi_i$ as
\begin{equation}
\phi_i=\phi_{i,1}+\phi_{i,2}
\end{equation}
and $\phi_{i,1} $ is even in $y^i_1$ and
\begin{equation}
\|\phi_{i,2}\|_{**}\leq C||\frac{s_i-s_{i-1}}{\ve}|-|\frac{s_{i}-s_{i+1}}{\ve}||\min\{w(\frac{s_i-s_{i-1}}{\ve}),w(\frac{s_i-s_{i+1}}{\ve})\}\leq C\ve^2.
\end{equation}
Then by the above estimate and the decomposition in Proposition \ref{pro2}, we have

\begin{equation}\label{sym}
\int_{\Omega_\ve}\chi_\ve(z-P_i)\phi_iZ_i dz=O(\ve^2),
\end{equation}
and similar to the decomposition of $\phi_i$, one can also decompose $\phi_{i-1}+\phi_{i+1}$ as an even function of $y^i_1$ and an $O(\ve^2)$ function, so we get that
\begin{eqnarray*}
&&\sum_{j=i-1,i+1}\int_{\Omega_\ve}\chi_\ve(z-P_j)\phi_jZ_idz\\
&&=\int_{\R^2_+}(\phi_{i-1}+\phi_{i+1})\frac{\partial w(y^i)}{\partial y_1}dy+O(\ve^2)\\
&&=O(\ve^2)
\end{eqnarray*}
Moreover, since $|s_1-s_2|=2(1+o(1))|\ve\ln\ve|$ and $|s_{k-1}-s_k|=2(1+o(1))|\ve\ln\ve|$, one can get that
\begin{equation}
\int_{\Omega_\ve}\chi_\ve(z-P_2)\phi_2Z_1dz=O(\ve^2), \int_{\Omega_\ve}\chi_\ve(z-P_{k-1})\phi_{k-1}Z_kdz=O(\ve^2).
\end{equation}
Thus we have
\begin{equation}
|s_i|\leq C\ve^2.
\end{equation}
Next we estimate $\psi_{13}$. Denote by
$$f_{\ve}=L_\ve(\psi-\sum_{i=1}^k\chi_\ve(z-p_i)\phi_i-\ve^2(\psi_{11}+\psi_{12})) .$$

\noindent
Claim:
\begin{equation}
\|f_\ve\|_*\leq C\ve^2.
\end{equation}

\noindent
Proof of the Claim:

By the definition of $f_{\ve}$, we have
\begin{eqnarray*}
&&f_\ve(z)=L_\ve(\psi-\sum_{i=1}^k\chi_\ve(z-P_i)\phi_i-\ve^2(\psi_{11}-\psi_{12})) \\
&&=E+N(\psi)+\sum_{i}c_iZ_i-\sum_i L_\ve(\chi_\ve(z-P_i)\phi_i)-\ve^2L(\psi_{11}+\psi_{12})\\
&&=(\sum_i P_{\Omega_{\ve,P_i}}w(z-P_i))^p-\sum_i w(z-P_i)^p+N(\psi)+\sum_{i}c_iZ_i\\
&&-\sum_i \chi_\ve(z-P_i)(\Delta_y \phi_i-\phi_i+p((\sum_i P_{\Omega_\ve}w(z-P_i))^{p-1}+O(\ve))\phi_i)\\
&&+\sum_i(2\nabla \phi_i \nabla(\chi_\ve(z-P_i))+\phi_i\Delta \chi_\ve(z-P_i))-\ve^2L_\ve(\psi_{11}+\psi_{12})\\
&&=(\sum_i P_{\Omega_{\ve,P_i}}w(z-P_i))^p-\sum_i w(z-P_i)^p+N(\psi)+\sum_{i}c_iZ_i\\
&&-\sum_i \chi_\ve(z-P_i)(p((\sum_i P_{\Omega_{\ve,P_i}}w(z-P_i))^{p-1}-w(y-P_i)^{p-1})\phi_i\\
&&-pw(y-P_i)^{p-1}(w(y-P_{i-1})+w(y-P_{i+1})+\ve v_{1i}(y))+d_iZ_i)\\
&&+\sum_i O(\ve)\phi_i+\sum_i(2\nabla \phi_i \nabla(\chi_\ve(z-p_i))+\phi_i\Delta \chi_\ve(z-P_i))-\ve^2L_\ve(\psi_{11}+\psi_{12}).
\end{eqnarray*}

From the definition and estimates of $\phi_i$, $\psi_{11}, \psi_{12}$, $\chi$, and the configuration space, we know that $|c_i|=O(\ve^2)$, so
\begin{equation*}
\|f_\ve\|_*\leq  C\ve^2.
\end{equation*}

By the a priori estimate, we know that
\begin{equation*}
\|\psi_{13}\|_*\leq C,
\end{equation*}
thus we have
\begin{equation*}
\|\psi_1\|_*\leq C.
\end{equation*}
We thus finish the proof.

\end{proof}

\medskip
Given  points $P_j$  defined by \equ{points1}, Proposition
\ref{p3} guarantees the existence (and gives estimates) of a unique
solution $\psi$, $c_i$, $i=1, \ldots , k$, to Problem \equ{nonlinear}.  It is clear then
that the function $u= U + \psi$ is an exact solution to our problem
\equ{p}, with the required properties stated in Theorem \ref{teo} if
we show that there exists a configuration for the points $P_j$  that gives all the constants $c_i$  in \equ{nonlinear}
equal to zero. In order to do so we first need to find the correct
conditions on the points to get $c_i=0$. This condition is
naturally given by projecting in $L^2 (\Omega_\ve )$ the equation in
\equ{nonlinear} into the space spanned by $Z_i$, namely by
multiplying the equation in \equ{nonlinear} by $Z_i$  and
integrate all over $\Omega_\ve$. We will do it in details in the next final Section.

\setcounter{equation}{0}
\section{The Reduced Problem}\label{sec6}
In this section, we keep the notations and assumptions in the previous sections. As explained in the previous section, we have obtained a solution $u=\sum_{i=1}^kP_{\Omega_{\ve,P_i}}w(z-P_i)+\sum_{i=1}^k\chi_\ve(z-P_i)\phi_i+\ve^2\psi_1$ of the following equation
\begin{equation}\label{e501}
\left\{\begin{array}{l}
\Delta u-u+u^p=\sum_{i=1}^k c_i Z_i \mbox{ in }\Omega_\ve\\
\frac{\partial u}{\partial \nu}=0 \mbox{ on }\partial \Omega_\ve
\end{array}
\right.
\end{equation}

In this section, we are going to solve $c_i=0$ for all $i$ by adjusting the position of the spikes, i.e. $P_i$. First, multiplying the above equation (\ref{e501}) by $Z_i, i=1,\cdots, k$ and integrating over $\Omega_\ve$, we get that
\begin{equation}
M\left[\begin{array}{c}
c_1\\
c_2\\
\vdots\\
c_k
\end{array}
\right]=\left[\begin{array}{c}
\int_{\Omega_\ve}(\Delta u-u+u^p)Z_1\\
\int_{\Omega_\ve}(\Delta u-u+u^p)Z_2\\
\vdots\\
\int_{\Omega_\ve}(\Delta u-u+u^p)Z_k
\end{array}
\right]
\end{equation}

Recall that $M$ is invertible, so $c_i=0, i=1,\cdots,k$ is reduced to solve the following system:
\begin{equation}\label{reduce}
\left[\begin{array}{c}
\int_{\Omega_\ve}(\Delta u-u+u^p)Z_1\\
\int_{\Omega_\ve}(\Delta u-u+u^p)Z_2\\
\vdots\\
\int_{\Omega_\ve}(\Delta u-u+u^p)Z_k
\end{array}
\right]=0.
\end{equation}

We have the following estimates:

\begin{lemma}\label{lemma601}
Under the assumption of Proposition \ref{p3}, for $\ve$ small enough, the following expansion holds:
\begin{equation}
\int_{\Omega_\ve}(\Delta u-u+u^p)Z_1dz=-\Psi(\frac{s_1-s_2}{\ve})-\ve^2\nu_2H'(\gamma(s_1))+O(\ve^3),
\end{equation}
and for $i=2,\cdots,k-1$,
\begin{equation}
\int_{\Omega_\ve}(\Delta u-u+u^p)Z_idz=\Psi(\frac{s_i-s_{i-1}}{\ve})-\Psi(\frac{s_i-s_{i+1}}{\ve})-\ve^2\nu_2H'(\gamma(s_i))+O(\ve^3)
\end{equation}
and
\begin{equation}
\int_{\Omega_\ve}(\Delta u-u+u^p)Z_kdz=\Psi(\frac{s_k-s_{k-1}}{\ve})-\ve^2\nu_2H'(\gamma(s_k))+O(\ve^3).
\end{equation}
where $\nu_2>0$ is a constant defined in (\ref{curvature}).
\end{lemma}

\medskip

\begin{proof}
First, by direct calculation, one can get the following expansion:
\begin{eqnarray*}
&&\Delta u-u+u^p\\
&&=\Big[\Delta (U+\sum_{i=1}^k\chi_\ve(z-P_i)\phi_i)-(U+\sum_{i=1}^k\chi_\ve(z-P_i)\phi_i)+(U+\sum_i\chi_\ve(z-P_i)\phi_i)^p\Big]\\
&&+\Big[\ve^2(\Delta\psi_1-\psi_1+p(U+\sum_{i=1}^k\chi_\ve(z-P_i)\phi_i)^{p-1}\psi_1)\Big]\\
&&+\Big[(U+\sum_{i=1}^k\chi_\ve(z-P_i)\phi_i+\ve^2\psi_1)^p-(U+\sum_{i=1}^k\chi_\ve(z-P_i)\phi_i)^p-\\
&&p(U+\sum_{i=1}^k\chi_\ve(z-P_i)\phi_i)^{p-1}\ve^2\psi_1\Big]\\
&&:=I_1+I_2+I_3.
\end{eqnarray*}

Next we calculate $I_1$ to $I_3$ term by term. First from the estimate on $\psi_1$ in (\ref{psi})
\begin{eqnarray}\label{i2}
&&\int_{\Omega_\ve}I_2Z_i=\ve^2\int_{\Omega_\ve}(\Delta \psi_1-\psi_1+p(U+\sum_i\chi_\ve(z-p_i)\phi_i)^{p-1}\psi_1)Z_i\nonumber\\
&&=\ve^2\int_{\Omega_\ve}-pw(z-P_i)^{p-1}\frac{\partial w(z-P_i)}{\partial \tau}\psi_1+p(U+\sum_i\chi_\ve(z-P_i)\phi_i)^{p-1}Z_i\psi_1\nonumber\\
&&=\ve^2\int_{\Omega_\ve}p(p-1)w(z-P_i)^{p-2}\frac{\partial w(z-P_i)}{\partial \tau}\psi_1(\sum_{j\neq i}\frac{\partial w(z-P_j)}{\partial \tau}+O(\ve))dz\nonumber\\
&&=O(\ve^3).
\end{eqnarray}
Moreover
\begin{eqnarray}\label{i3}
\int_{\Omega_\ve}I_3Z_i&=&\int_{\Omega_\ve}\Big[(U+\sum_i\chi_\ve(z-P_i)\phi_i+\ve^2\psi_1)^p-(U+\sum_i\chi_\ve(z-P_i)\phi_i)^p\nonumber\\
&&-p(U+\sum_i\chi_\ve(z-P_i)\phi_i)^{p-1}\ve^2\psi_1\Big]Z_i\nonumber\\
&\leq & C\int_{\Omega_\ve}\ve^4|\psi_1|^2|Z_i|
=O(\ve^3).
\end{eqnarray}
Next from the equation satisfied by $\phi_i$ and the definition of the cutoff function $\chi$, we get that
\begin{eqnarray}\label{i1i1}
\int_{\Omega_\ve}I_1Z_i&=&\int_{\Omega_\ve}(\Delta U-U+U^p)Z_i+\sum_j\int_{\Omega_\ve}\chi_\ve(z-P_j)(\Delta\phi_j-\phi_j+pU^{p-1}\phi_j)Z_i\nonumber\\
&&+[(U+\sum_i\chi_\ve(z-P_i)\phi_i)^p-U^p-pU^{p-1}\sum_i \chi_\ve(z-p_i)\phi_i]Z_i+O(\ve^3)\nonumber\\
&=&\int_{\Omega_\ve}(\Delta U-U+U^p)Z_i+I_{11}+I_{12}+O(\ve^3).
\end{eqnarray}
Similar to the estimate in (\ref{sym}), using the equation satisfied by $\phi_i$, we have for $i=2,\cdots,k-1$
\begin{eqnarray*}
&&\sum_j\int_{\Omega_\ve}\chi_\ve(z-P_j)(\Delta\phi_j-\phi_j+pU^{p-1}\phi_j)Z_i\\
&&=\int_{\Omega_\ve}\chi_\ve(z-P_i)(\Delta\phi_j-\phi_j+pU^{p-1}\phi_j)Z_i\\
&&+\sum_{j\neq i}\int_{\Omega_\ve}\chi_\ve(z-P_j)(\Delta\phi_j-\phi_j+pU^{p-1}\phi_j)Z_i\\
&&=\int_{\Omega_\ve}p(U^{p-1}-w_i^{p-1})\phi_iZ_i+O(\ve^3)\\
&&+\sum_{j=i-1,i+1}\int_{\Omega_\ve}\chi(z-P_j)(\Delta\phi_j-\phi_j+pU^{p-1}\phi_j)Z_i+O(\ve^3)\\
&&=\int_{\Omega_\ve}p(p-1)w_i^{p-2}(w_{i+1}+w_{i-1})\phi_iZ_i\\
&&+\sum_{j=i-1,i+1}\int_{\Omega_\ve}\chi_\ve(z-P_j)(\Delta\phi_j-\phi_j+pU^{p-1}\phi_j)Z_i+O(\ve^3)\\
&&=O(\ve||\frac{s_i-s_{i-1}}{\ve}|-|\frac{s_i-s_{i+1}}{\ve}||\min\{w(\frac{s_i-s_{i-1}}{\ve}), w(\frac{s_i-s_{i+1}}{\ve})\})+O(\ve^3)\\
&&=O(\ve^3)
\end{eqnarray*}
and similarly we can always decompose
\begin{equation*}
\sum_{j=1}^k\chi_\ve(z-P_j)\phi_j=\ve\psi_{1,i}+O(\ve^2)
\end{equation*}
where $\psi_{1,i}$ is  a function even in $y^i_1$, and by Proposition \ref{pro2}, we have
\begin{equation*}
Z_i=\frac{\partial w(y^i)}{\partial y_1}+\ve\eta_i+O(\ve^2)
\end{equation*}
where $\eta_i$ is odd in $y^i_1$. Thus we have
\begin{eqnarray*}
I_{12}&&\leq C\int_{\Omega_\ve}p(p-2)w_i^{p-2}(\sum_{j=1}^k\chi_\ve(z-P_j)\phi_j)^2Z_idz+O(\ve^3)\\
&&\leq C\ve^3.
\end{eqnarray*}
For the case $i=1,k$, recall that $w(\frac{s_1-s_2}{\ve}), w(\frac{s_k-s_{k-1}}{\ve})=O(\ve^2)$, one can also get that
\begin{equation}\label{i1i2}
I_{11}+I_{12}=O(\ve^3).
\end{equation}
Thus we have the following:
\begin{equation*}
\int_{\Omega_\ve}I_1Z_idz=\int_{\Omega_\ve}(\Delta U-U+U^p)Z_i+O(\ve^3).
\end{equation*}

Next for $i=2,\cdots,k-1$
\begin{eqnarray}\label{i11}
&&\int_{\Omega_\ve}(\Delta U-U+U^p)Z_i\\
&&=\int_{\Omega_\ve}\Big[(\sum_i P_{\Omega_{\ve,P_i}}w(z-P_i))^p-\sum_i w(z-P_i)^p\Big]Z_i\nonumber\\
&&=\int_{\Omega_\ve}\Big[(w(z-P_i)+\ve v_i^{(1)}+\ve^2(v_i^{(2)}+v_i^{(3)})\nonumber\\
&&+\sum_{j\neq i}P_{\Omega_{\ve,P_i}}w(z-P_j)+O(\ve^3))^p-\sum_iw(z-P_i)^p\Big]Z_i\nonumber\\
&&=\int_{\Omega_\ve}pw(z-P_i)^{p-1}\Big(\ve v_i^{(1)}+\ve^2(v_i^{(2)}+v_i^{(3)})\nonumber\\
&&+w(z-P_{i-1})+w(z-P_{i+1})\Big)\frac{\partial w(z-P_i)}{\partial \tau}+O(\ve^3)\nonumber\\
&&=\int_{\R^2_+}pw(y)(w(y-\frac{s_{i-1}-s_i}{\ve}e_1)+w(y-\frac{s_{i+1}-s_i}{\ve}e_1))\frac{\partial w(y)}{\partial y_1}\nonumber\\
&&+\ve^2\int_{\R^2_+}pw(y)^{p-1}\frac{\partial w(y)}{\partial y_1}v_i^{(3)}+O(\ve^3)\nonumber
\end{eqnarray}

Similarly, one has for $i=1,k$,
\begin{eqnarray*}
&&\int_{\Omega_\ve}(\Delta U-U+U^p)Z_1\\
&&=\int_{\R^2_+}pw(y)w(y-\frac{s_{2}-s_1}{\ve}e_1)\frac{\partial w(y)}{\partial y_1}\nonumber\\
&&+\ve^2\int_{\R^2_+}pw(y)^{p-1}\frac{\partial w(y)}{\partial y_1}v_1^{(3)}+O(\ve^3),
\end{eqnarray*}
and
\begin{eqnarray*}
&&\int_{\Omega_\ve}(\Delta U-U+U^p)Z_k\\
&&=\int_{\R^2_+}pw(y)w(y-\frac{s_{k-1}-s_k}{\ve}e_1)\frac{\partial w(y)}{\partial y_1}\nonumber\\
&&+\ve^2\int_{\R^2_+}pw(y)^{p-1}\frac{\partial w(y)}{\partial y_1}v_k^{(3)}+O(\ve^3).
\end{eqnarray*}

Next by the definition of $v_i^{(3)}$, we can get that
\begin{eqnarray}\label{curvature}
\int_{\R^2_+}pw(y)^{p-1}\frac{\partial w(y)}{\partial y_1}v_i^{(3)}dy&=&\int_{\R^2_+}-(\Delta -1)\frac{\partial w(y)}{\partial y_1}v_i^{(3)}\nonumber\\
&=-&\int_{\partial \R^2_+}\frac{\partial w(y)}{\partial y_1}\frac{\partial v_i^{(3)}}{\partial y_2}-v_i^{(3)}\frac{\partial }{\partial y_2}\frac{\partial w(y)}{\partial y_1}dy\nonumber\\
&=&-\frac{1}{3}\int_{\R}(\frac{w'(|y|)}{|y|})^2\rho^{(3)}(P_i)y_1^4dy_1\nonumber\\
&=&-\nu_2 \rho^{(3)}(P_i)=-\nu_2 H'(\gamma(s_i)),
\end{eqnarray}
where $\nu_2=\frac{1}{3}\int_{\R}(\frac{w'}{|y|})^2y_1^4>0$ is a positive constant.

Recall that the interaction function is defined by
\begin{equation}\label{Psi}
\Psi(s)=-\int_{\R^2_+}pw(y-(s,0))w(y)^{p-1}\frac{\partial w(y)}{\partial y_1},
\end{equation}

Combining (\ref{i1i1}), (\ref{i1i2}),(\ref{i11}), (\ref{curvature}) and (\ref{Psi}), we know that
\begin{eqnarray}\label{i111i}
\int_{\Omega_\ve}I_1Z_1 dz=-\Psi(|\frac{s_1-s_2}{\ve}|)-\ve^2 \nu_2H'(\gamma(s_1))+O(\ve^3)
\end{eqnarray}
and for $i=2,\cdots,k-1$
\begin{eqnarray}\label{i1111}
\int_{\Omega_\ve}I_1Z_i dz=\Psi(|\frac{s_i-s_{i-1}}{\ve}|)-\Psi(|\frac{s_i-s_{i+1}}{\ve}|)-\ve^2\nu_2 H'(\gamma(s_i))+O(\ve^3)
\end{eqnarray}
and
\begin{eqnarray}\label{i111k}
\int_{\Omega_\ve}I_1Z_k dz=\Psi(|\frac{s_k-s_{k-1}}{\ve}|)-\ve^2 \nu_2H'(\gamma(s_i))+O(\ve^3).
\end{eqnarray}
The results follows from (\ref{i2}), (\ref{i3}) and (\ref{i111i})-(\ref{i111k}).

\end{proof}

From Lemma \ref{lemma601},  the problem (\ref{reduce}) is reduced to the following system:
\begin{equation*}
\left\{\begin{array}{l}
\Psi_1(|\frac{s_1-s_2}{\ve}|)+\ve^2H'(\gamma(s_1))=O(\ve^3),\\
\Psi_1(|\frac{s_3-s_2}{\ve}|)-\Psi(|\frac{s_2-s_1}{\ve}|)+\ve^2H'(\gamma(s_2))=O(\ve^3),\\
\ \ \ \ \ \   \ \  \vdots\\
\Psi_1(|\frac{s_k-s_{k-1}}{\ve}|)-\Psi(|\frac{s_{k-1}-s_{k-2}}{\ve}|)+\ve^2H'(\gamma(s_{k-1}))=O(\ve^3),\\
-\Psi_1(|\frac{s_k-s_{k-1}}{\ve}|)+\ve^2H'(\gamma(s_k))=O(\ve^3).
\end{array}
\right.
\end{equation*}
where we denote by
\begin{equation}
\Psi_1(s)=\nu_2^{-1}\Psi(s).
\end{equation}
\medskip

By summing up the first $i$ equations, one has
\begin{equation}\label{reducedproblem}
\left\{\begin{array}{l}
\Psi_1(\frac{s_{i+1}-s_i}{\ve})+\sum_{j=1}^i\ve^2 H'(\gamma(s_j))=O(i\ve^3) \mbox{ for }i=1,\cdots,k-1,\\
\sum_{i=1}^k\ve^2H'(\gamma(s_i))=O(k\ve^3).
\end{array}
\right.
\end{equation}

\setcounter{equation}{0}
\section{Solving the nonlinear system}\label{sec7}

Our aim in the rest of this paper is to find a solution $\{s_i\}$ to the non-linear system (\ref{reducedproblem}) in (\ref{points1}).

\medskip

Observe that the linearized matrix of the above system at main order is degenerate, thus the terms containing $H'(\gamma(s))$ will play an important role. We will explain how we solve system (\ref{reducedproblem}). The novelty of this paper is to consider the above system as a discretization of an ODE. In order to explain this idea, we first introduce some notations.

\medskip

Let
\begin{equation*}
s=G(b)
\end{equation*}
be the solution of $\Psi_1(s)=b$. Since $\Psi_1(s)=C_n s^{-\frac{1}{2}}e^{-s}(1+o(1))$ as $s\to \infty$, using this asymptotic behaviour of $\Psi_1$, one has the following:
\begin{equation}
G(b)=-(1+O(\frac{\ln(-\ln b)}{\ln b}))\ln b, \mbox{ as } b\to 0.
\end{equation}

Then the above reduced system (\ref{reducedproblem}) is equivalent to the following system:
\begin{equation}\label{e3}
\left\{\begin{array}{l}
s_{i+1}-s_i=\ve G(-\sum_{j=1}^i\ve^2 H'(\gamma(s_j))+O(\ve^3 i)), \mbox{ for }i=1,\cdots,k-1\\
s_k-s_{k-1}=\ve G(\ve^2H'(\gamma(s_k))+O(\ve^3k)).
\end{array}
\right.
\end{equation}

Let $h=-\ve\ln\ve$ be the boot size, if we denote by $s_i=x(t_i)$ where $t_i=(i-1)h$, then from the above system (\ref{e3}),
\begin{equation}\label{e16}
\left\{\begin{array}{l}
\frac{x(t_{i+1})-x(t_i)}{h}=-\frac{1}{\ln\ve}G(-\frac{\ve}{\ln\ve}(-\sum_{j=1}^iH'(\gamma(x(t_j)))h)+O(\ve^3i)),\\
\frac{x(t_k)-x(t_{k-1})}{h}=-\frac{1}{\ln\ve}G(\ve^2H'(\gamma(x(t_k)))+O(\ve^3k)).
\end{array}
\right.
\end{equation}

In order the solve the above system, we consider the limiting case of the above system, i.e. view $\frac{x(t_{i+1})-x(t_i)}{h}$ as $x'(t)$ and $\sum_{j=1}^iH'(\gamma(x(t_j)))h$ as $\int_{0}^tH'(\gamma(x(t)))dt$, and introduce the following ODE:
\begin{equation}\label{ode}
\left\{\begin{array}{l}
\frac{dx}{dt}=-\frac{1}{\ln\ve}G(\frac{\ve}{\ln\ve}\rho(t)),\\
\frac{d\rho}{dt}=H'(\gamma(x(t))),\\
\rho(0)=0,  \ \rho(b_\ve)=\rho_b,\\
x'(b_\ve)=-\frac{1}{\ln\ve}G(\ve^2H'(\gamma(x(b_\ve)))),
\end{array}
\right.
\end{equation}
where $b_\ve=(k-1)h=[\frac{b}{h}]h=b+O(h)$.

\medskip

One can see that the above second order ODE has three initial conditions. Besides the two end point initial values,  there is an extra condition, i.e. the last equation of (\ref{ode}), which in fact comes from the last equation of (\ref{e3}). This ODE with extra initial condition is not always solvable. It turns out that this extra condition corresponds to some balancing condition of the curvature of the segment $\gamma$. In order to solve this ODE, we need assumption $(H_1)$ on $\gamma$. For this ODE, we have the following existence result:
\begin{lemma}
Under the assumption $(H_1)$, there exists $\ve_0>0$, such that for every $\ve<\ve_0$, there exist  $\rho_b=\rho_b(\ve)<0$, such that the above ODE (\ref{ode}) is solvable. Moreover, $\rho_b$ satisfies the following asymptotic behaviour:
\begin{equation}
\rho_b=-(H'(\gamma(b_\ve))+O(\frac{\ln(-\ln\ve)}{\ln\ve}))h.
\end{equation}
\end{lemma}

\begin{proof}
From the asymptotic behaviour of $G$, we know that the first equation of (\ref{ode}) is
\begin{eqnarray*}
\frac{dx}{dt}&=&-\frac{1}{\ln\ve}G(\frac{\ve}{\ln\ve}\rho(t))\\
&=&(1+O(\frac{\ln(-\ln\ve)}{\ln\ve}))(a_1\ln(-\rho(t))+a_2)
\end{eqnarray*}
where
\begin{equation*}
a_1=\frac{1}{\ln\ve}, \ a_2=1-\frac{\ln(-\ln\ve)}{\ln\ve}.
\end{equation*}
Integrating the above equation from $b_\ve$ to $t$, one has
\begin{eqnarray*}
x(t)-x(b_\ve)&=&\int_{b_\ve}^t-\frac{1}{\ln\ve}G(\frac{\ve}{\ln\ve}\rho(t))dt\\
&=&(1+O(\frac{\ln(-\ln\ve)}{\ln\ve}))[a_2(t-b_\ve)+a_1\int_{b_\ve}^t\ln(-\rho(t))dt].
\end{eqnarray*}
Plugging the expression for $x(t)$ into the second equation,
\begin{equation}
\rho'(t)=H'(\gamma(x(b_\ve))+(1+O(\frac{\ln(-\ln\ve)}{\ln\ve}))[a_2(t-b_\ve)+a_1\int_{b_\ve}^t\ln(-\rho(t))dt]).
\end{equation}
By the boundary condition $\rho(0)=0, \rho(b_\ve)=\rho_b$, we have
\begin{equation}\label{e11}
\int_{b_\ve}^0H'(\gamma(x(b_\ve)+(1+O(\frac{\ln(-\ln\ve)}{\ln\ve}))[a_2(t-b_\ve)+a_1\int_{b_\ve}^t\ln(-\rho(t))dt]))dt=-\rho_b.
\end{equation}
By Taylor's expansion,
\begin{eqnarray*}
&&H'(\gamma(x(b_\ve)+(1+O(\frac{\ln(-\ln\ve)}{\ln\ve}))[a_2(t-b_\ve)+a_1\int_b^t\ln(-\rho(t))dt]))\\
&&=H'(\gamma(x(b_\ve)+a_2(t-b)+O(\frac{\ln(-\ln\ve)}{\ln\ve})))\\
&&=H'(\gamma(x(b_\ve)+a_2(t-b_\ve)))+O(\frac{\ln(-\ln\ve)}{\ln\ve}).
\end{eqnarray*}
So from (\ref{e11}) and the above equation, we have
\begin{eqnarray}\label{e1}
\int_{b_\ve}^0H'(\gamma(x(t)))dt&=&\int_{b_\ve}^0H'(\gamma(x(b_\ve)+a_2(t-b)))dt+O(\frac{\ln(-\ln\ve)}{\ln\ve})\nonumber\\
&=&H(\gamma(x(b_\ve)-a_2b_\ve))-H(\gamma(x(b_\ve)))+O(\frac{\ln(-\ln\ve)}{\ln\ve})\nonumber\\
&=&\rho_b.
\end{eqnarray}

Since by the third boundary condition
\begin{equation}
x'(b_\ve)=-\frac{1}{\ln\ve}G(\ve^2H'(\gamma(x(b_\ve)))),
\end{equation}
one can get that
\begin{equation}
\rho_b=H'(\gamma(x(b_\ve)))\ve\ln\ve.
\end{equation}
We assume that
\begin{equation}
\rho_b=(H'(\gamma(b_\ve))+\rho_\ve)\ve\ln\ve,
\end{equation}
then
\begin{equation}\label{e12}
x(b_\ve)=b_\ve+\frac{(1+o(1))\rho_\ve}{H''(\gamma(b_\ve))}.
\end{equation}

Using (\ref{e12}),  (\ref{e1}) is reduced to the following:
\begin{eqnarray}
&&H(\gamma(0))-H(\gamma(b_\ve))+\frac{H'(\gamma(0))-H'(\gamma(b_\ve))}{H''(\gamma(b_\ve))}\rho_\ve\\
&&+o(\rho_\ve)+O(\rho_\ve^2)=O(\frac{\ln(-\ln\ve)}{\ln\ve})\nonumber.
\end{eqnarray}

By the assumption $(H_1)$
\begin{equation}
H(\gamma(0))=H(\gamma(b)), \ H'(\gamma(0))\neq H'(\gamma(b)),
\end{equation}
and
\begin{equation}
H''(\gamma)\geq c_0>0, \ b_\ve=b+O(h),
\end{equation}
the above equation is uniquely solvable with
\begin{equation}
\rho_\ve=O(\frac{\ln(-\ln\ve)}{\ln\ve}).
\end{equation}

So there exists unique $\rho_b=(H'(\gamma(b_\ve))+O(\frac{\ln(-\ln\ve)}{\ln\ve}))\ve\ln\ve$ such that (\ref{ode}) is solvable, and we have
\begin{equation}
x(0)=O(\frac{\ln(-\ln\ve)}{\ln\ve}), \ x(b_\ve)=b_\ve+O(\frac{\ln(-\ln\ve)}{\ln\ve}).
\end{equation}

\end{proof}

We will use the solution of the ODE to approximate the solution of (\ref{e3}). In order to obtain a good approximate solution, one need to control the error of
\begin{equation*}
\sum_{j=1}^iH'(\gamma(x(t_j)))h-\int_{0}^{t_{i+1}}H'(\gamma(x(t)))dt.
\end{equation*}
So we will use the midpoint Riemann sum approximation of integrals which will give us
\begin{equation}
\sum_{j=1}^iH'(\gamma(x(t_j)))h-\int_{0}^{t_{i+1}}H'(\gamma(x(t)))dt=O(h^2).
\end{equation}
To be more specific, we will choose the approximate solution to be the following:
\begin{equation}
x_i^0=x(\bar{t}_i), \  \bar{t}_i=\frac{t_i+t_{i+1}}{2},  i=1,\cdots, k-1,
\end{equation}
and
\begin{equation}
x_k^0=x_{k-1}^0+\ve G(\frac{\ve}{\ln\ve}\rho_b)
\end{equation}
where $x(t)$ is the solution determined by the ODE (\ref{ode}).

We want to find the solution to (\ref{e3}) of the form
\begin{equation}
s_i=x_i^0+y_i.
\end{equation}
Then  $y_i$ will satisfy the following equation:
\begin{equation}\label{e6}
\left\{\begin{array}{l}
y_{i+1}-y_i=-E_i+\ve\Big( G(-\ve^2\sum_{j=1}^iH'(\gamma(x_j^0+y_j))+O(\ve^3i))-G(-\ve^2\sum_{j=1}^iH'(\gamma(x_j^0))) \Big),\\
\ \ \ \ \ \ \ \ \ \ \mbox{ for }i=1,\cdots,k-1\\
\ve^2\sum_{j=1}^kH''(\gamma(x_j^0))y_j+O(\ve^2)\sum_{j=1}^k|y_j|^2=-E_k+O(\ve^3 k),
\end{array}
\right.
\end{equation}
where
\begin{equation*}
E_i=x_{i+1}^0-x_i^0-\ve G(-\ve^2\sum_{j=1}^iH'(\gamma(x_j^0)))
\end{equation*}
for $i=1,\cdots,k-1$, and
\begin{equation*}
E_k=\ve^2\sum_{j=1}^kH'(\gamma(x_j^0)).
\end{equation*}

First we show that the approximate solution we choose is indeed a good approximate solution, i.e. the error $E_i$ is small enough. In fact, we have the following error estimate:
\begin{lemma}
\begin{equation}
E_i=x_{i+1}^0-x_i^0-\ve G(-\ve^2\sum_{j=1}^iH'(\gamma(x_j^0)))=O(\ve)
\end{equation}
for $i=1,\cdots,k-1$, and
\begin{equation}
E_k=\ve^2\sum_{j=1}^kH'(\gamma(x_j^0))=O(\ve^2\frac{\ln(-\ln\ve)}{\ln\ve}).
\end{equation}
Moreover, the following estimate holds:
\begin{equation}
\sum_{i=1}^{k-1}|E_i|=O(\ve).
\end{equation}
\end{lemma}
\begin{proof}
First for $i=k-1$, we have
\begin{eqnarray*}
&&x_{k}^0-x_{k-1}^0-\ve G(-\ve^2\sum_{j=1}^{k-1}H'(\gamma(x_j^0)))\\
&&=\ve G(\frac{\ve}{\ln\ve}\rho_b)-\ve G(-\ve^2\sum_{j=1}^{k-1}H'(\gamma(x_j^0)))\\
&&=O(\frac{\ve}{\rho_b})|\rho_b-\sum_{j=1}^{k-1}H'(\gamma()x_j^0)h|
\end{eqnarray*}

Since we choose the midpoint approximation, we have for $i=1,\cdots,k-2$,
\begin{equation}\label{e4}
\rho(t_{i+1})-\sum_{j=1}^iH'(\gamma(x_j^0))h=O(h^2),
\end{equation}
and
\begin{eqnarray}\label{e5}
\sum_{j=1}^kH'(\gamma(x_j^0))h&=&(\sum_{j=1}^{k-1}H'(\gamma(x_j^0))h-\rho(t_k))+(H'(\gamma(x_k^0))h+\rho(t_k))\nonumber\\
&=&O(h^2)+O(\frac{\ln(-\ln\ve)}{\ln\ve})h=O(\frac{\ln(-\ln\ve)}{\ln\ve})h.
\end{eqnarray}
By (\ref{e4}) and (\ref{e5}), and recall that $\rho_b=O(h)$, one can obtain that
\begin{equation*}
E_{k-1}=x_{k}^0-x_{k-1}^0-\ve G(-\ve^2\sum_{j=1}^{k-1}H'(\gamma(x_j^0)))=O(\ve h)
\end{equation*}
and
\begin{equation}
E_k=O(\ve^2\frac{\ln(-\ln\ve)}{\ln\ve}).
\end{equation}

Next by the equation satisfied by $\rho(t)$, we can get that
\begin{equation}
\rho(t_i)=O(\min\{i,k-i+1\}h)
\end{equation}
so for $i=1,\cdots,k-2$
\begin{eqnarray*}
&&x_{i+1}^0-x_i^0-\ve G(-\ve^2\sum_{j=1}^iH'(\gamma(x_j^0)))\\
&&=\int_{\bar{t}_i}^{\bar{t}_{i+1}}-\frac{1}{\ln\ve}G(\frac{\ve}{\ln\ve}\rho(t))dt-\ve G(-\ve^2\sum_{j=1}^iH'(\gamma(x_j^0)))\\
&&=-\frac{1}{\ln\ve}G(\frac{\ve}{\ln\ve}\rho(t_{i+1}))h-\ve G(-\ve^2\sum_{j=1}^iH'(\gamma(x_j^0)))+O(\frac{\rho''\rho-(\rho')^2}{|\ln\ve|\rho^2}(t_{i+1}))h^3\\
&&=\ve(G(\frac{\ve}{\ln\ve}\rho(t_{i+1}))-G(-\ve^2\sum_{j=1}^iH'(\gamma(x_j^0))))+O(\frac{\rho''\rho-(\rho')^2}{|\ln\ve| \rho^2}(t_{i+1}))h^3\\
&&=O(\frac{\ve}{\rho(t_{i+1})})(\rho(t_{i+1})-\sum_{j=1}^iH'(\gamma(x_j^0))h)+O(\frac{\rho''\rho-(\rho')^2}{|\ln \ve|\rho^2}(t_{i+1}))h^3\\
&&=O(\frac{\ve h}{\min\{i,k-i+1\}})+O(\ve)(\frac{1}{\min\{i,k-i+1\}^2}+\frac{h}{\min\{i,k-i+1\}})\\
&&=O(\ve).
\end{eqnarray*}

Moreover, from the above estimate, we have
\begin{equation*}
\sum_{j=1}^iE_j=O(\ve), \mbox{ for }i=1,\cdots,k-1.
\end{equation*}

\end{proof}

Finally, we will show that equation (\ref{e6}) is solvable.
\begin{lemma}
There exists $\ve_0>0$, such that for $\ve<\ve_0$, there exists a solution $\{y_i\}_{1\leq i\leq k}$ to (\ref{e6}) such that
\begin{equation}
\|y\|_\infty\leq C\ve\ln(-\ln\ve).
\end{equation}
\end{lemma}

\begin{proof}
For $\|y\|_\infty<< \ve|\ln \ve|$, we have
\begin{eqnarray*}
&&\ve G(-\ve^2\sum_{j=1}^iH'(\gamma(x_j^0+y_j))+O(\ve^3 i))-\ve G(-\ve^2\sum_{j=1}^iH'(\gamma(x_j^0)))\\
&&=-\ve (\frac{\sum_{j=1}^iH''(\gamma(x_j^0))y_j}{\sum_{j=1}^iH'(\gamma(x_j^0))})+O(\frac{\ve i |y|_{j\leq i}^2}{\sum_{j=1}^iH'(\gamma(x_j^0))})+O(\frac{\ve^2 i}{\sum_{j=1}^iH'(\gamma(x_j^0))})
\end{eqnarray*}

The equations (\ref{e6}) for $y_i$ can be rewritten as follows:
\begin{equation}\label{e13}
\left\{\begin{array}{l}
y_{i+1}-y_i+\ve\frac{\sum_{j=1}^iH''(\gamma(x_j^0))y_j}{\sum_{j=1}^iH'(\gamma(x_j^0))}=-E_i+O(\frac{\ve i |y|_{j\leq i}^2}{\sum_{j=1}^iH'(\gamma(x_j^0))})+O(\frac{\ve^2 i}{\sum_{j=1}^iH'(\gamma(x_j^0))})\\
\ \ \ \ \ \ \ \ \ \ \ \ \ \ \ \ \ \ \ \ \ \ \mbox{ for }i=1,\cdots,k-1\\
\sum_{j=1}^kH''(\gamma(x_j^0))y_j+\sum_{j=1}^kH'''(\gamma(x_j^0))y_j^2=O(\ve k)+O(\frac{\ln(-\ln\ve)}{\ln\ve}).
\end{array}
\right.
\end{equation}
We will show that one can first solve $y_2$ to $y_k$ in terms of $y_1$ from the first $k-1$ equations, and finally solve $y_1$ by the $k$-th equation of (\ref{e13}).

For $1\leq l \leq i_0=(1-\delta)k$ where $\delta>0$ is a small number to be determined later, we have
\begin{eqnarray*}
&&y_{l+1}-y_1+\ve\sum_{i=1}^l\frac{\sum_{j=1}^iH''(\gamma(x_j^0))y_j}{\sum_{j=1}^iH'(\gamma(x_j^0))}\\
&&=\sum_{i=1}^lE_i+\sum_{i=1}^l\frac{\ve i}{\sum_{j=1}^iH'(\gamma(x_j^0))}|y|^2_{i\leq l}+\sum_{i=1}^l\frac{\ve^2 i}{\sum_{j=1}^iH'(\gamma(x_j^0))}\\
&&=O(\ve)+\sum_{i=1}^l\frac{\ve i|y|^2_{i\leq l}}{\min\{i,k-i+1\}}+\sum_{i=1}^lO(\frac{\ve^2 i}{\min\{i,k-i+1\}})\\
&&=O(\ve)+O(\frac{\ve l}{\delta})|y|^2_{i\leq l}.
\end{eqnarray*}
where we denote by
\begin{equation*}
|y|_{i_1\leq i\leq i_2}=\sup_{i_1\leq i\leq i_2}|y_i|.
\end{equation*}
Moreover,
\begin{eqnarray*}
\ve\sum_{i=1}^l\frac{\sum_{j=1}^iH''(\gamma(x_j^0))y_j}{\sum_{j=1}^iH'(\gamma(x_j^0))}
&=&\ve\sum_{i=1}^lO(\frac{i|y|_{i\leq l}}{\min\{i,k-i+1\}})\\
&=&O(\frac{\ve l |y|_{i\leq l}}{\delta})=o(1)|y|_{i\leq l}.
\end{eqnarray*}

Thus one can get that for $l\leq i_0$
\begin{equation}
y_l=y_1+o(1)|y|_{i\leq i_0}+o(1)|y|^2_{i\leq i_0}+O(\ve)
\end{equation}

So we can get that
\begin{equation}
y_i=(1+o(1))y_1+O(\ve),  i=2,\cdots,i_0.
\end{equation}

For $l > i_0$, we have the following:
\begin{eqnarray*}
y_{l+1}-y_1&=&-\ve \sum_{i=i_0+1}^l\frac{\sum_{j=1}^iH''(\gamma(x_j^0))y_j}{\sum_{j=1}^iH'(\gamma(x_j^0))}\\
&+&O(\ve)+O(|y|^2_{i_0< i\leq l})+O(\frac{\ve l}{\delta})|y|^2_{i\leq i_0}+o(1)|y|_{i\leq i_0}\\
&=&C_0\delta |y|_{i_0< i\leq l}+O(|y|^2_{i_0< i\leq l})+O(|y|_{i\leq i_0})+O(\ve)\\
\end{eqnarray*}
for some $C_0$ independent of $\ve$ and $\delta$.
So for $i_0< i \leq k$,
\begin{equation}
y_i=O(y_1)+C_0\delta |y|_{i_0< i\leq l}+O(|y|^2_{i_0< i\leq l})+O(\ve)
\end{equation}

If $\delta>0$ is small such that $C_0\delta <\frac{1}{4}$, then the above system is solvable with
\begin{equation}
y_i=O(y_1)+O(\ve).
\end{equation}
From the last equation, we have
\begin{eqnarray*}
&&\sum_{i=1}^{i_0}H''(\gamma(x_i^0))y_i+\sum_{i=i_0+1}^kH''(\gamma(x_i^0))y_i+O(k|y_1|^2)+O(k\ve^2)\\
&&=(\sum_{i=1}^{i_0}H''(\gamma(x_i^0))(1+o(1))y_1+O(\delta k |y_1|)+O(k\ve)+
O(k|y_1|^2)\\
&&=O(\frac{\ln(-\ln\ve)}{\ln\ve}).
\end{eqnarray*}
Thus by the assumption $(H_1)$, the equation is reduced to
\begin{eqnarray*}
y_1=o(1)y_1+O(\delta)|y_1|+O(|y_1|^2)+O(\ve\ln(-\ln\ve))
\end{eqnarray*}
If we further choose $\delta $ small enough but independent of $\ve$ such that $O(\delta)|y_1|<\frac{1}{2}|y_1|$, it is easy to see that by contraction mapping, the above equation has a solution and satisfies
\begin{equation}
y_1=O(\ve\ln(-\ln\ve)).
\end{equation}
Thus we get that there exists a solution to (\ref{e6}) with
\begin{equation*}
\|y\|_\infty\leq C\ve\ln(-\ln\ve)<<\ve|\ln\ve|.
\end{equation*}

Thus we have proved the existence of solution to (\ref{e6}).
\end{proof}

\bigskip

\setcounter{equation}{0}
\section{Appendix: Proof of Proposition \ref{p1}}

In this appendix, we shall give a proof of Proposition \ref{p1}. The proof is rather standard. It follows from arguments in \cite{awz} and \cite{lnw}. It is
based on Fredholm Alternative Theorem for compact
operator and an a-priori estimates.

First we need an estimate on the following matrix $M$ defined by
\begin{equation}
M_{ij}=\int_{\Omega_\ve}Z_i Z_j dz, \ i,j=1,\cdots,k.
\end{equation}

\begin{lemma}\label{lemma1}
For $\ve$ sufficiently small, given any vector $\vec{b}\in \R^k$, there exists a unique vector $\vec{\beta}\in \R^k$, such that $M\vec{\beta}=\vec{b}$. Moreover,
\begin{equation}\label{matrix}
\|\vec{\beta}\|_\infty\leq C\|\vec{b}\|_\infty
\end{equation}
for some constant $C$ independent of $\ve$.
\end{lemma}
\begin{proof}
To prove the existence, it is sufficient to prove the a priori estimate (\ref{matrix}). Suppose that $|\beta_i|=\|\beta\|_\infty$, we have
\begin{equation*}
\sum_{i=1}^kM_{ij}\beta_j=b_i.
\end{equation*}
For the entries $M_{ij}$, from the definition of $\Lambda_k$, and the exponential decay property of $Z_i$, we know that
\begin{equation*}
M_{ii}=\int_{\Omega_\ve}Z_i^2dz=(1+o(1))\int_{\R^2_+}(\frac{\partial w}{\partial y_1})^2dy>c_0>0,
\end{equation*}
and
\begin{equation*}
\sum_{j\neq i}|M_{ij}|\leq C\sum_{j\neq i}e^{-\frac{|P_i-P_j|}{2}}=o(1).
\end{equation*}
Hence for $\ve$ small, we have
\begin{equation*}
c_0\|\vec{\beta}\|_\infty\leq c_0|\vec{\beta}_i|\leq \sum_{j\neq i}|M_{ij}||\vec{\beta}_j|+|b_i|\leq o(1)\|\vec{\beta}\|_\infty+\|\vec{b}\|_\infty
\end{equation*}
from which the desired result follows.
\end{proof}

Second we need the following a priori estimate

\begin{lemma} \label{p2}
Let $h\in L^2(\Omega_\ve)$ with $\| h \|_* $ bounded and
assume that $(\psi , \{c_i\} )$ is a solution to \equ{linear}.
Then there exist positive numbers $\ve_0$ and $C$, such that for all $\ve \leq \ve_0$, for any
points $P_i$, $i=1,\ldots , k$
given by \equ{points1} , one has
\begin{equation}
 \|\psi \|_{*} \le C \|h\|_*  .
\label{est1}\end{equation}
\end{lemma}

\begin{proof} We argue by contradiction. Assume there exist $\psi $ solution
to \equ{linear} and
$$
\| h \|_* \to 0, \quad \| \psi \|_{*} =1.
$$
We prove that
\begin{equation} \label{cazero}
c_i \to 0  \mbox{ for }i=1,\cdots,k.
\end{equation}
Multiply the equation in \equ{linear} against $Z_{j}$ and integrate in
$\Omega_\ve$, we get
$$
\int_{\Omega_\ve} L_\ve \psi Z_{j} (z)  = \int_{\Omega_\ve}  h Z_{j}  + M(c_{j}) ,
$$
By the exponentially decay of $Z_i$, we first know that
$$
|\int_{\Omega_\ve}  h Z_j | \leq C \| h \|_* .
$$
Here and in what follows, $C$ stands for a positive constant
independent of $\ve$, as $\ve \to 0$ .  Secondly, by the equation satisfied by $P_{\Omega_{\ve,P_i}}w(z-P_i)$,
\begin{eqnarray*}
&&\int_{\Omega_\ve}L_\ve\psi Z_idz=\int_{\Omega_\ve}(\Delta \psi-\psi+p(\sum_{i=1}^kP_{\Omega_\ve}w(z-P_i))^{p-1}\psi)Z_idz\\
&&=\int_{\Omega_\ve}(\Delta Z_i-Z_i+p(\sum_{i=-k}^kP_{\Omega_{\ve,P_i}}w(z-P_i))^{p-1}Z_i)\psi dz\\
&&=\int_{\Omega_\ve}p(\sum_{i=1}^k[P_{\Omega_{\ve,P_i}}w(z-P_i))^{p-1}\frac{\partial P_{\Omega_{\ve,P_i}}w(z-P_i)}{\partial \tau}-pw(z-P_i)^{p-1}\frac{\partial w(z-P_i)}{\partial \tau}]\psi dz\\
&&\leq C\int_{B_{|\ln\ve|}(P_i)}|\frac{\partial w(z-P_i)}{\partial \tau}||O(\ve)+w(z-P_i)^{p-2}\sum_{j\neq i}P_{\Omega_{\ve,P_i}}w(z-P_j)||\psi|dz\\
&&+\int_{\Omega_\ve/B_{|\ln\ve|}(P_i)}\frac{\partial |w(z-p_i)}{\partial \tau}|[\sum_{j=1}^k w_j^{p-1}+O(\ve)\sum_{j=1}^ke^{-\mu|z-P_j|}]|\psi|dz\\
&&\leq C\|\psi\|_*(O(\ve)+O(\ve^{\frac{p-\eta}{2}}))\\
&&\leq C\ve\|\psi\|_*
\end{eqnarray*}
if we choose $\eta$ small enough such that $p-\eta>2$. This can be done since $p>2$.

Since $M$ is invertible and $\|M^{-1}\|\leq C$, we get that
\begin{equation}
|c_i|\leq C(\|h\|_*+O(\ve)\|\psi\|_*).
\end{equation}

Thus we get the validity of \equ{cazero}, since we are assuming $\|
\psi \|_{*} =1$ and $ \| h \|_* \to 0$.

\medskip
Let now $\mu \in (0,1)$. It is easy to check that the function
\[
W : =  \sum_{i=-k}^{k} e^{-\mu \, |\cdot - P_i|}  ,
\]
satisfies
\[
L_\ve\, W  \leq  \frac{1}{2} \,  ( \mu^2 -1) \, W \, ,
\]
in $\Omega_\ve \setminus \cup_{j=1, \ldots , k} B(P_j , R)$ provided $R$ is fixed large enough (independently of $\ve$).  Hence the function $W$ can be used as a barrier to prove the pointwise estimate
\begin{equation}
|\phi | (x)  \leq  C \, \left( \|  L_\ve  \, \psi \|_*  + \sum_{j}  \|×\psi\|_{L^\infty(B(p_j, R)\cap \Omega_\ve)} \right) \,   W (x) \, ,
\label{eq:fp2}
\end{equation}
for all $z \in \Omega_\ve \setminus \cup_{j} B(P_j , R)$.

\medskip

Granted these preliminary estimates, the proof  of the result goes by contradiction. Let us assume there exist  a sequence of $\ve\to 0$ and a sequence of solutions of  \equ{linear}  for which the inequality is not true. The problem being linear, we can reduce to the case where we have a sequence $\ve^{(n)}$ tending to $0$ and sequences $ h^{(n)}$, $\psi^{(n)}, c^{(n)}$ such that
$$
\| h^{(n)} \|_* \to 0, \quad \mbox{and} \quad \| \psi^{(n)} \|_{*} =1.
$$
But  \equ{cazero} implies that we also have
\[
\| c^{(n)} \|_* \to 0 \, .
\]
Then \equ{eq:fp2} implies that there exists $P_i^{(n)} $ such that
\begin{equation}\label{eqfp3}
\|\psi^{(n)} \|_{L^\infty (B(P_i^{(n)} ,R))}\geq C ,
\end{equation}
for some fixed constant $C>0$. Using elliptic estimates together with Ascoli-Arzela's theorem, we can find a sequence $P_i^{(n)}$ and we can extract, from the sequence $\psi_i^{(n)} (\cdot -P_i^{(n)})$ a subsequence which will converge (on compact) to $\psi_\infty$ a solution of
\[
\left(\Delta - 1 + p \, w^{p-1} \right) \, \psi_\infty =0 \, ,
\]
in $\R^2_+$, which is bounded by a constant times $e^{-\mu \, |x|}$, with $\mu >0$. Moreover, since $\psi_i^{(n)}$ satisfies the orthogonality conditions in \equ{linear},  the limit function $\psi_\infty$ also satisfies
\[
\int_{\R^2_+} \psi_\infty \,  \frac{\partial w}{\partial y_1} \, dx =0 \, .
\]
But the solution $w$ being non-degenerate, this implies that $\psi_\infty \equiv 0$, which is certainly in contradiction with \equ{eqfp3} which implies that $\psi_\infty$ is not identically equal to $0$.
\medskip

Having reached a contradiction, this completes the proof of the Lemma.

\end{proof}

\medskip
We can now prove Proposition \ref{p1}.

\medskip
\noindent {\it Proof of Proposition \ref{p1}.} Consider the space
$$
{\mathcal H} = \{ u \in H_N^2 (\Omega_\ve ) \, : \, \int_{\Omega_\ve} u Z_i = 0 , \quad  i=1, \ldots , k
\}.
$$
Notice that the problem \equ{linear} in $\psi $ gets re-written as
\begin{equation}\label{lp7}
\psi + K (\psi ) = \bar h \quad {\mbox{in}} \quad {\mathcal H}
\end{equation}
where $\bar h$ is defined by duality and $K: {\mathcal H} \to
{\mathcal H}$ is a linear compact operator. Using Fredholm's
alternative, showing that equation \equ{lp7} has a unique solution
for each $\bar h$ is equivalent to showing that the equation has a
unique solution for $\bar h = 0$, which in turn follows from
Proposition \ref{p2}. The estimate \equ{est} follows directly from
Proposition \ref{p2}. This concludes the proof of Proposition
\equ{p1}.


\begin{thebibliography}{99}




 \bibitem{amn} Ambrosetti, A; Malchiodi A; Ni, W.-M.  Singularly
   perturbed elliptic equations with symmetry: existence of solutions
   concentrating on spheres, Part II. {\em  Indiana Univ. Math. J. }
 {\bf 53} (2004), no. 2, 297-329.


 \bibitem{amw}Ao, W., Musso, M.  and Wei, J.   On spikes concentrating on line
segments to a semilinear Neumann problem, {\em Journal of Differential
Equations}, 251(2011), no. 4-5, 881-901.
\bibitem{amw1}Ao, W. , Musso, M. and Wei,  J. Triple junction solutions for a singularly perturbed Neumann problem, {\em SIAM Journal on Mathematical Analysis}, 43(2011), no. 6, 2519-2541.

\bibitem{awz} Ao, W.W., Wei, J.C. and Zeng, J.  An optimal bound on the number of
interior spike solutions for the Lin-Ni-Takagi problem, {\em Journal of
Functional Analysis}, 265(2013), no.7, 1324-1356.

\bibitem{bf} Bates, P.; Fusco, G.  Equilibria with many nuclei for
the Cahn-Hilliard equation. {\em J. Diff. Eqns}
{\bf 160} (2000), 283-356.

\bibitem{bds} Bates, P.; Dancer, E.N.; Shi, J.  Multi-spike stationary solutions
 of the
Cahn-Hilliard  equation in higher-dimension and instability.
 {\em Adv. Diff. Eqns } {\bf 4} (1999), 1-69.

\bibitem{bm} Butscher, A. and Mazzeo, R. CMC hypersurfaces condensing to geodesic segments and rays in Riemannian Manifolds, {\em Ann. Sc. Norm. Super. Pisa. Cl. Sci. } (5), Vol. XI (2012), 653-706.


\bibitem{dap} T. D'Aprile and A. Pistoia, On the existence of some new positive interior spike solutions to a semilinear Neumann problem, {\em J. Diff. Eqns.} 248(2010), 556-573.


\bibitem{bu} A. Butscher, CMC surfaces in Riemannian manifolds Condensing to a compact network of curves, arXiv:0910.4442.


\bibitem{dfw1} del Pino, M.; Felmer, P.; Wei, J.-C.  On the role of
mean curvature in some singularly perturbed Neumann problems. {\em
SIAM J. Math. Anal.} {\bf 31} (1999), 63-79.

\bibitem{dfw2}  del Pino, M.; Felmer, P.; Wei, J.-C.  On the role of
distance function in some singularly perturbed  problems.
 {\em  Comm. PDE} {\bf 25} (2000), 155-177.

\bibitem{dfw3} del Pino, M.; Felmer, P.; Wei, J.-C.   Mutiple peak
solutions for some singular perturbation problems. {\em  Cal. Var. PDE} {\bf 10}
 (2000), 119-134.



\bibitem{dy} Dancer, E.N.; Yan, S.  Multipeak solutions for a singular perturbed
 Neumann problem. {\em Pacific J. Math.} {\bf 189} (1999),
241-262.

\bibitem{fm}Mahmoudi, F. and Malchiodi, A. Concentration on minimal sub manifolds for a singularly perturbed Neumann problem. {\em Adv. in Math}, 209-2 (2007), 460-525.

\bibitem{fm1}Mahmoudi,F. and Malchiodi, A. Concentration at manifolds of arbitrary dimension for a singularly perturbed Neumann problem. {\em Atti Accad. Naz. Lincei Cl. Sci. Fis. Mat. Natur. Rend. Lincei (9) Mat. Appl.} 17 (2006), no. 3, 279Ð290.
\bibitem{gm}
A. Gierer and H. Meinhardt,
{\it A theory of biological pattern formation}, Kybernetik (Berlin) 12 (1972), 30-39.

\bibitem{gw1} Gui, C.-F.; Wei, J.-C.  Multiple interior spike solutions
for some singular perturbed Neumann problems. {\em J. Diff. Eqns.}
{\bf 158} (1999), 1-27.

\bibitem{gw2}  Gui, C.-F.; Wei, J.-C.  On multiple mixed interior and boundary
peak solutions for some singularly perturbed Neumann problems.  {\em Can. J. Math. } {\bf 52} (2000), 522-538.

\bibitem{gww} Gui, C.-F.; Wei, J.-C.; Winter, M.  Multiple boundary peak
solutions for some singularly perturbed Neumann problems.
{\em Ann. Inst. H. Poincar\'{e} Anal. Non Lin\'{e}aire} {\bf 17} (2000), 249-289.


\bibitem{gpw}  Grossi, M.; Pistoia, A.; Wei, J.-C.  Existence of multi-peak solutions for a semi-linear Neumann problem via non-smooth critical point theory.
 {\em Cal.    Var.  PDE}  {\bf 11} (2000), 143-175.


\bibitem{k} Kwong, M.K.  Uniquness of positive solutions of $\Delta u-u + u^p=0$
in $\R^N$. {\em Arch. Rational Mech. Anal.} {\bf 105} (1991), 243-266.

\bibitem{ks} E.F. Keller and L.A. Segel,
{\it Initiation of slime mold aggregation viewed as an instability}, J. Theor. Biol. 26, 1970, 399-415.

\bibitem{li} Li, Y.-Y. On a singularly perturbed equation with Neumann boundary condition.  {\em Comm. P.D.E.} {\bf 23} (1998), 487-545.

\bibitem{ln} Li, Y.-Y.; Nirenberg, L.  The Dirichlet problem for singularly perturbed elliptic equations.  {\em Comm. Pure Appl. Math.} {\bf 51} (1998),  1445-1490.

\bibitem{lnw} F.H. Lin, W.M. Ni and J.C. Wei, On the number of interior peak solutions for a singularly perturbed Neumann problem, {\em Comm. Pure Appl. Math.} 60 (2007), 252-281.



\bibitem{m}  Malchiodi, A.  Concentration at curves for a singularly perturbed  Neumann problem in three-dimensional domains. {\em Geom. Funct. Anal.} 15(2005), no.6, 1162-1222.


\bibitem{ma} Malchiodi, A.,  Some new entire solutions of semilinear elliptic equations in $\R^N$, {\em Adances in Math.} (2009), no. 6, 1843-1909.




\bibitem{mm1} Malchiodi, A.; Montenegro, M. Boundary concentration phenomena for a singularly perturbed ellptic problem.  {\em Comm. Pure Appl. Math.} {\bf 55} (2002), 1507-1508.

\bibitem{mm2} Malchiodi, A.; Montenegro, M.  Multidimensional boundary layers for a singularly perturbed Neumann problem. {\em  Duke Math. J.} {\bf 124} (2004), no. 1, 105-143.

\bibitem{mm3} Boundary layers of arbitrary dimension for a singularly perturbed Neumann problem. {\em Mat. Contemp.} 27 (2004), 117Ð146.


\bibitem{mnw}  Malchiodi, A.; Ni, W.-M.; Wei, J.-C. Multiple clustered layer solutions for semi-linear Neumann problems on a ball. {\em Ann. Inst. H. Poincar\'{e} Anal. Non Lin\'{e}aire} {\bf 22} (2005), no.2, 143-163.




\bibitem{ni} Ni, W.-M.  Diffusion, cross-diffusion, and their spike-layer steady states. {\em Notices of Amer. Math. Soc.} {\bf  45} (1998), 9-18.


\bibitem{nisurvey}
W.-M. Ni, {\em Qualitative properties of solutions to elliptic problems.}  Stationary partial differential equations. Vol. I,  157--233, Handb. Differ. Equ., North-Holland, Amsterdam, 2004.




\bibitem{nt1}  Ni, W.-M.; Takagi, I.  On the shape of least energy solution to a semilinear Neumann problem. {\em Comm. Pure    Appl. Math.} {\bf 41} (1991), 819-851.

\bibitem{nt2} Ni, W.-M.; Takagi, I.  Locating the peaks of least energy solutions to a semilinear Neumann problem. {\em Duke Math. J. } {\bf 70} (1993), 247-281.

\bibitem{nw} Ni, W.-M.; Wei, J.-C.   On the location and profile of spike-layer solutions to singularly perturbed semi-linear Dirichlet problems. {\em Comm. Pure Appl. Math.}  {\bf 48} (1995), 731-768.



\bibitem{wboundary} Wei, J.-C.  On the boundary spike layer solutions of singularly perturbed semilinear Neumann problem. {\em J. Diff. Eqns.} {\bf 134} (1997), 104-133.





\bibitem{weisurvey} J. Wei, {\em Existence and Stability of Spikes for the Gierer-Meinhardt System}, in  HANDBOOK OF DIFFERENTIAL EQUATIONS, Stationary Partial Differential Equations, volume 5 (edited by M. Chipot), pp. 487-585.  North-Holland, Amsterdam, 2008.


\bibitem{ww1} Wei, J.-C.; Winter, M.   Stationary solutions for the Cahn-Hilliard equation,
{\em  Ann. Inst. H. Poincar\'{e} Anal. Non Lin\'{e}aire} {\bf 15} (1998), 459-492.

\bibitem{ww2} Wei, J.-C.; Winter, M.  Multiple boundary spike solutions for a wide class of singular perturbation problems. {\em J. London Math. Soc.} {\bf 59} (1999), 585-606.


\bibitem{wy} Wei, J.-C.; Yang, Jun, Concentration on lines for a singularly perturbed Neumann problem in two-dimensional domains, {\em Indiana Univ. Math. J.} 56(2007), 3025-3073.

\bibitem{wy1}Wei, Juncheng; Yang, Jun Toda system and interior clustering line concentration for a singularly perturbed Neumann problem in two dimensional domain. Discrete Contin. Dyn. Syst. 22 (2008), no. 3, 465Ð508



\end{thebibliography}
\end{document}